\documentclass[11pt]{article}%
\usepackage{amssymb}
\usepackage{amsfonts}
\usepackage{amsmath}
\usepackage{amssymb}
\usepackage{color}
\usepackage{graphicx}
\usepackage{caption}%
\providecommand{\U}[1]{\protect\rule{.1in}{.1in}}
\newtheorem{theorem}{Theorem}

\newtheorem{corollary}[theorem]{Corollary}

\newtheorem{definition}[theorem]{Definition}

\newtheorem{remark}[theorem]{Remark}

\newenvironment{proof}[1][Proof]{\noindent \textbf{#1.} }{\  \rule{0.5em}{0.5em}}
\begin{document}

\title{\textbf{A Wald-type test statistic for testing linear hypothesis in logistic
regression models based on minimum density power divergence estimator}}
\author{Basu, A.$^{1}$; Ghosh, A.$^{2}$; Mandal, A.$^{3}$; Martin, N.$^{4}$ and Pardo,
L.$^{4}$\\$^{1}${\small Indian Statistical Institute, Kolkata 700108, India}\\$^{2}${\small University of Oslo, Oslo, Norway}\\$^{3}${\small University of Pittsburgh, Pittsburgh, PA 15260, USA}\\$^{4}${\small Complutense University of Madrid, 28040 Madrid, Spain} }
\date{\today }
\maketitle

\begin{abstract}
In this paper a robust version of the classical Wald test statistics for
linear hypothesis in the logistic regression model is introduced and its
properties are explored. We study the problem under the assumption of random
covariates although some ideas with non random covariates are also considered.
The family of tests considered is based on the minimum density power
divergence estimator instead of the maximum likelihood estimator and it is
referred to as the Wald-type test statistic in the paper. We obtain the
asymptotic distribution and also study the robustness properties of the Wald
type test statistic. The robustness of the tests is investigated theoretically
through the influence function analysis as well as suitable practical
examples. It is theoretically established that the level as well as the power
of the Wald-type tests are stable against contamination, while the classical
Wald type test breaks down in this scenario. Some classical examples are
presented which numerically substantiate the theory developed. Finally a
simulation study is included to provide further confirmation of the validity
of the theoretical results established in the paper.

\end{abstract}

\begin{center}
\bigskip
\end{center}



\section{Introduction\label{sec1}}

Experimental settings often include dichotomous response data, wherein a
Bernoulli model may be assumed for the independence response variables
$Y_{1},...,Y_{n}$, with
\[
\Pr(Y_{i}=1)=\pi_{i}\text{ and }\Pr(Y_{i}=0)=1-\pi_{i},\text{ }i=1,...,n.
\]
In many cases, a series of explanatory variables $x_{i0},...,x_{ik}$ may be
associated with each $Y_{i}$ ($x_{i0}=1,$ $x_{ij}\in\mathbb{R}$, $i=1,...,n$,
$j=1,...,k$, $k<n$). We shall assume that the binomial parameter, $\pi_{i} $,
is linked to the linear predictor $\sum_{j=0}^{k}\beta_{j}x_{ij}$ via the
logit function, i.e.,
\begin{equation}
\text{\textrm{logit}}\left(  \pi_{i}\right)  =\sum\limits_{j=0}^{k}\beta
_{j}x_{ij},\label{logit}%
\end{equation}
where \textrm{logit}$(p)=\log(p/(1-p))$. In the following, we shall denote the
binomial parameter $\pi_{i}$, by
\begin{equation}
\pi_{i}=\pi(\boldsymbol{x}_{i}^{T}\boldsymbol{\beta})=\frac{e^{\boldsymbol{x}%
_{i}^{T}\boldsymbol{\beta}}}{1+e^{\boldsymbol{x}_{i}^{T}\boldsymbol{\beta}}%
},\text{ }i=1,...,n,\label{1.1}%
\end{equation}
where\ $\boldsymbol{x}_{i}^{T}=\left(  x_{i0},...,x_{ik}\right)  $ and
$\boldsymbol{\beta}=\left(  \beta_{0},...,\beta_{k}\right)  ^{T}$ is a $(k+1)
$-dimensional vector of unknown parameters with $\beta_{i}\in\left(
-\infty,\infty\right)  $. The \textquotedblleft design
matrix\textquotedblright, $\mathbb{X}=\left(  \boldsymbol{x}_{1}%
,...,\boldsymbol{x}_{n}\right)  ^{T}$, is assumed to be full rank
(\emph{rank}$\left(  \mathbb{X}\right)  =k+1$), without any loss of generality.

Let $\boldsymbol{M}$ be any matrix of $r$ rows and $k+1$ columns with
$\mathrm{rank}(\boldsymbol{M})=r$, and $\boldsymbol{m}$ a vector of order $r$
with specified constants such that $\mathrm{rank}(\boldsymbol{M}%
^{T},\boldsymbol{m})=r$. If we are interested in testing
\begin{equation}
H_{0}:\boldsymbol{M}^{T}\boldsymbol{\beta}=\boldsymbol{m},\label{1.2}%
\end{equation}
the Wald test statistic is usually used in which $\boldsymbol{\beta}$ is
estimated using the maximum likelihood estimator (MLE). Notice that if we
consider $\boldsymbol{M}=\boldsymbol{I}_{k+1}$ and $\boldsymbol{m}%
=\boldsymbol{\beta}_{0}$, we get the Wald-type test statistic presented by
Bianco and Martinez (2009) based on a weighted Bianco and Yohai (1996)
estimator. It is well known that the MLE of $\boldsymbol{\beta}$ can be
severely affected by outlying observations. Croux et al. (2002) discuss the
breakdown behavior of the MLE in the logistic regression model and show that
the MLE breaks down when several outliers are added to a data set. In the
recdent years several authors have attempted to derive robust estimates of the
parameters in the logistic regression model; see for instance Pregibon (1982),
Morgenthaler (1992), Carrol and Pedersen (1993), Cristmann (1994), Bianco and
Yohai (1996), Croux and Haesbroeck (2003), Bondell (2005, 2008) and Hobza et
al. (2012). Our interest in this paper is to present a family of Wald-type
test statistics based on the robust minimum density power divergence estimator
for testing the general linear hypothesis given in (\ref{1.2}).

In Section \ref{sec2} we present the minimum density power divergence
estimator for $\boldsymbol{\beta}$. The Wald-type test statistics, based on
the minimum density power divergence estimator, are presented in Section
\ref{sec3}, as well as their asymptotic properties. The theoretical robustness
properties are presented in Section \ref{sec4} and finally, Section \ref{sec5}
and \ref{sec6}\ are devoted to present a simulation study and real data
examples, respectively.

\section{Minimum density power divergence estimator\label{sec2}}

If we denote by $y_{_{1},},...,y_{n}$ the observed values of the random
variables $Y_{1},...,Y_{n},$ the likelihood function for the logistic
regression model is given by
\begin{equation}
\mathcal{L}\left(  \boldsymbol{\beta}\right)  =\prod\limits_{i=1}^{n}%
\pi^{y_{i}}(\boldsymbol{x}_{i}^{T}\boldsymbol{\beta})\left(  1-\pi
(\boldsymbol{x}_{i}^{T}\boldsymbol{\beta})\right)  ^{1-y_{i}}.\label{1.3}%
\end{equation}
So the MLE of $\boldsymbol{\beta}$, $\widehat{\boldsymbol{\beta}}$, is
obtained minimizing log-likelihood function almost surely over
$\boldsymbol{\beta}$ belonging to%
\[
\Theta=\left\{  \left(  \beta_{0},...,\beta_{k}\right)  ^{T}:\beta_{j}%
\in\left(  -\infty,\infty\right)  ,\text{ }j=0,...,k\right\}  =\mathbb{R}%
^{k+1}.
\]

We consider the probability vectors,
\[
\widehat{\boldsymbol{p}}=\left(  \frac{y_{1}}{n},\frac{1-y_{1}}{n},\frac
{y_{2}}{n},\frac{1-y_{2}}{n},...,\frac{y_{n}}{n},\frac{1-y_{n}}{n}\right)
^{T}%
\]
and
\[
\boldsymbol{p}\left(  \boldsymbol{\beta}\right)  =\left(  \pi(\boldsymbol{x}%
_{1}^{T}\boldsymbol{\beta})\frac{1}{n},\left(  1-\pi(\boldsymbol{x}_{1}%
^{T}\boldsymbol{\beta})\right)  \frac{1}{n},...,\pi(\boldsymbol{x}_{n}%
^{T}\boldsymbol{\beta})\frac{1}{n},\left(  1-\pi(\boldsymbol{x}_{n}%
^{T}\boldsymbol{\beta})\right)  \frac{1}{n}\right)  ^{T}.
\]
The Kullback-Leibler divergence measure between the probability vectors
$\widehat{\boldsymbol{p}}$ and $\boldsymbol{p}\left(  \boldsymbol{\beta
}\right)  $ is given by
\begin{equation}
d_{KL}\left(  \widehat{\boldsymbol{p}},\boldsymbol{p}\left(  \boldsymbol{\beta
}\right)  \right)  =\sum\limits_{i=1}^{n}\sum\limits_{j=1}^{2}\frac{y_{ij}}%
{n}\log\frac{y_{ij}}{\pi_{j}(\boldsymbol{x}_{i}^{T}\boldsymbol{\beta}%
)},\label{1.5}%
\end{equation}
where%
\[
\pi_{1}(\boldsymbol{x}_{i}^{T}\boldsymbol{\beta})=\pi(\boldsymbol{x}_{i}%
^{T}\boldsymbol{\beta})\text{, }\pi_{2}(\boldsymbol{x}_{i}^{T}%
\boldsymbol{\beta})=1-\pi(\boldsymbol{x}_{i}^{T}\boldsymbol{\beta}),\text{
}y_{i1}=y_{i}\text{ and }y_{i2}=1-y_{i}.
\]
It is not difficult to establish that
\begin{equation}
d_{KL}\left(  \widehat{\boldsymbol{p}},\boldsymbol{p}\left(  \boldsymbol{\beta
}\right)  \right)  =c-\frac{1}{n}\log\mathcal{L}\left(  \boldsymbol{\beta
}\right)  .\label{1.6}%
\end{equation}
Therefore, the MLE of $\boldsymbol{\beta}$ can be defined by
\begin{equation}
\widehat{\boldsymbol{\beta}}=\arg\min_{\boldsymbol{\beta\in}\Theta}%
d_{KL}\left(  \widehat{\boldsymbol{p}},\boldsymbol{p}\left(  \boldsymbol{\beta
}\right)  \right)  .\label{1.7}%
\end{equation}

Based on (\ref{1.7}) we can use any divergence measure $d\left(
\widehat{\boldsymbol{p}},\boldsymbol{p}\left(  \boldsymbol{\beta}\right)
\right)  $ in order to define a minimum divergence estimator for
$\boldsymbol{\beta}$. In this paper we shall use the density power divergence
measure defined by Basu et al. (1998) because the minimum density power
divergence estimators have excellent robustness properties, see for instance
Basu et al. (2011, 2013, 2015, 2016), Ghosh et al. (2015, 2016). The density
power divergence between the probability vectors $\widehat{\boldsymbol{p}}$
and $\boldsymbol{p}\left(  \boldsymbol{\beta}\right)  $ is given by
\begin{equation}
d_{\lambda}\left(  \widehat{\boldsymbol{p}},\boldsymbol{p}\left(
\boldsymbol{\beta}\right)  \right)  =\frac{1}{n^{1+\lambda}}\left\{
\sum\limits_{i=1}^{n}\left(  \sum\limits_{j=1}^{2}\pi_{j}^{1+\lambda
}(\boldsymbol{x}_{i}^{T}\boldsymbol{\beta})-\left(  1+\frac{1}{\lambda
}\right)  \sum\limits_{j=1}^{2}y_{ij}\pi_{j}^{\lambda}(\boldsymbol{x}_{i}%
^{T}\boldsymbol{\beta})\right)  +\frac{n}{\lambda}\right\} \label{1.4}%
\end{equation}
for $\lambda>0$. For $\lambda=0$, we have
\[
d_{0}\left(  \widehat{\boldsymbol{p}},\boldsymbol{p}\left(  \boldsymbol{\beta
}\right)  \right)  =\lim_{\lambda\rightarrow0}d_{\lambda}\left(
\widehat{\boldsymbol{p}},\boldsymbol{p}\left(  \boldsymbol{\beta}\right)
\right)  =d_{KL}\left(  \widehat{\boldsymbol{p}},\boldsymbol{p}\left(
\boldsymbol{\beta}\right)  \right)  .
\]

Based on (\ref{1.7}) and (\ref{1.4}), we shall define the minimum density
power divergence estimator in the following way.

\begin{definition}
The minimum density power divergence estimator for the parameter
$\boldsymbol{\beta}$, $\widehat{\boldsymbol{\beta}}_{\lambda}$, in the
logistic regression model is given by
\[
\widehat{\boldsymbol{\beta}}_{\lambda}=\arg\min_{\boldsymbol{\beta\in}\Theta
}d_{\lambda}\left(  \widehat{\boldsymbol{p}},\boldsymbol{p}\left(
\boldsymbol{\beta}\right)  \right)  ,
\]
where $d_{\lambda}\left(  \widehat{\boldsymbol{p}},\boldsymbol{p}\left(
\boldsymbol{\beta}\right)  \right)  $ was defined in (\ref{1.4}).
\end{definition}

In order to obtain the estimating equations we must get the derivative of
(\ref{1.4}) with respect to $\boldsymbol{\beta}$. First we are going to write
expression (\ref{1.4}) in the following way,%
\begin{align*}
d_{\lambda}\left(  \widehat{\boldsymbol{p}},\boldsymbol{p}\left(
\boldsymbol{\beta}\right)  \right)   &  =\frac{1}{n^{1+\lambda}}\left\{
\sum\limits_{i=1}^{n}\left(  \pi^{^{1+\lambda}}(\boldsymbol{x}_{i}%
^{T}\boldsymbol{\beta})+\left(  1-\pi(\boldsymbol{x}_{i}^{T}\boldsymbol{\beta
})\right)  ^{1+\lambda}\right.  \right. \\
&  \left.  -\left.  \left(  1+\frac{1}{\lambda}\right)  \left(  y_{i}%
\pi^{\lambda}(\boldsymbol{x}_{i}^{T}\boldsymbol{\beta})+(1-y_{i})\left(
1-\pi(\boldsymbol{x}_{i}^{T}\boldsymbol{\beta})\right)  ^{\lambda}\right)
\right)  +\frac{n}{\lambda}\right\}  .
\end{align*}
Now, taking into account that
\[
\frac{\partial\pi(\boldsymbol{x}_{i}^{T}\boldsymbol{\beta})}{\partial
\boldsymbol{\beta}}=\pi(\boldsymbol{x}_{i}^{T}\boldsymbol{\beta})\left(
1-\pi(\boldsymbol{x}_{i}^{T}\boldsymbol{\beta})\right)  \boldsymbol{x}%
_{i}\text{ and\ }\frac{\partial\left(  1-\pi(\boldsymbol{x}_{i}^{T}%
\boldsymbol{\beta})\right)  }{\partial\boldsymbol{\beta}}=-\pi(\boldsymbol{x}%
_{i}^{T}\boldsymbol{\beta})\left(  1-\pi(\boldsymbol{x}_{i}^{T}%
\boldsymbol{\beta})\right)  \boldsymbol{x}_{i}%
\]
and after some algebra, we get
\[
\frac{\partial d_{\lambda}\left(  \widehat{\boldsymbol{p}},\boldsymbol{p}%
\left(  \boldsymbol{\beta}\right)  \right)  }{\partial\boldsymbol{\beta}%
}=\frac{1+\lambda}{n^{\lambda+1}}\sum\limits_{i=1}^{n}(e^{\lambda
\boldsymbol{x}_{i}^{T}\boldsymbol{\beta}}+e^{\boldsymbol{x}_{i}^{T}%
\boldsymbol{\beta}})\frac{e^{\boldsymbol{x}_{i}^{T}\boldsymbol{\beta}}%
-y_{i}(1+e^{\boldsymbol{x}_{i}^{T}\boldsymbol{\beta}})}{(1+e^{\boldsymbol{x}%
_{i}^{T}\boldsymbol{\beta}})^{\lambda+2}}\boldsymbol{x}_{i}.
\]
Therefore, the estimating equations for $\lambda>0$ are given by
\begin{equation}
\sum\limits_{i=1}^{n}\frac{e^{\lambda\boldsymbol{x}_{i}^{T}\boldsymbol{\beta}%
}+e^{\boldsymbol{x}_{i}^{T}\boldsymbol{\beta}}}{(1+e^{\boldsymbol{x}_{i}%
^{T}\boldsymbol{\beta}})^{\lambda+1}}\left(  \pi(\boldsymbol{x}_{i}%
^{T}\boldsymbol{\beta})-y_{i}\right)  \boldsymbol{x}_{i}=\boldsymbol{0}%
,\label{1.61}%
\end{equation}
where $\pi(\boldsymbol{x}_{i}^{T}\boldsymbol{\beta})$ is (\ref{1.1}). Based on
the previous results we have established the following theorem.

\begin{theorem}
The minimum density power divergence estimator of $\boldsymbol{\beta}$,
$\widehat{\boldsymbol{\beta}}_{\lambda}$, can be obtained as the solution of
the system of equations given in (\ref{1.61}).
\end{theorem}

If we consider $\lambda=0$ in (\ref{1.61}), we get the estimating equations
for the MLE as%
\[
\sum\limits_{i=1}^{n}\left(  \pi(\boldsymbol{x}_{i}^{T}\boldsymbol{\beta
})-y_{i}\right)  \boldsymbol{x}_{i}=\boldsymbol{0}.
\]
Based on expression (\ref{1.61}), we can write the MDPDE for the logistic
regression model by
\[
\sum\limits_{i=1}^{n}\boldsymbol{\Psi}_{\lambda}\left(  \boldsymbol{x}%
_{i},y_{i},\boldsymbol{\beta}\right)  =\boldsymbol{0},
\]
with
\begin{equation}
\boldsymbol{\Psi}_{\lambda}\left(  \boldsymbol{x}_{i},y_{i},\boldsymbol{\beta
}\right)  =(e^{\lambda\boldsymbol{x}_{i}^{T}\boldsymbol{\beta}}%
+e^{\boldsymbol{x}_{i}^{T}\boldsymbol{\beta}})\frac{e^{\boldsymbol{x}_{i}%
^{T}\boldsymbol{\beta}}-y_{i}(1+e^{\boldsymbol{x}_{i}^{T}\boldsymbol{\beta}}%
)}{(1+e^{\boldsymbol{x}_{i}^{T}\boldsymbol{\beta}})^{\lambda+2}}%
\boldsymbol{x}_{i}.\label{1.100}%
\end{equation}
In order to get the asymptotic distribution of the MDPDE of $\boldsymbol{\beta
}$, $\widehat{\boldsymbol{\beta}}_{\lambda}$, we are going to assume that not
only the explanatory variables are random but are also identically distributed
and moreover
\[
\left(  \boldsymbol{X}_{1},Y_{1}\right)  ,....,\left(  \boldsymbol{X}%
_{n},Y_{n}\right)
\]
are independent and identically distributed. We shall assume that
$\boldsymbol{X}_{1},...,\boldsymbol{X}_{n}$ is a random sample from a random
variable $\boldsymbol{X}$ with marginal distribution function
$H(\boldsymbol{x})$. By following the method given in Maronna et al. (2006),
the asymptotic variance covariance matrix of $\sqrt{n}$
$\widehat{\boldsymbol{\beta}}_{\lambda}$ is
\[
\boldsymbol{J}_{\lambda}^{-1}\left(  \boldsymbol{\beta}_{0}\right)
\boldsymbol{K}_{\lambda}\left(  \boldsymbol{\beta}_{0}\right)  \boldsymbol{J}%
_{\lambda}^{-1}\left(  \boldsymbol{\beta}_{0}\right)  ,
\]
where
\[
\boldsymbol{K}_{\lambda}\left(  \boldsymbol{\beta}\right)  =E\left[
\boldsymbol{\Psi}_{\lambda}\left(  \boldsymbol{X},Y,\boldsymbol{\beta}\right)
\boldsymbol{\Psi}_{\lambda}^{T}\left(  \boldsymbol{X},Y,\boldsymbol{\beta
}\right)  \right]  =%
{\displaystyle\int_{\mathcal{X}}}
E\left[  \boldsymbol{\Psi}_{\lambda}\left(  \boldsymbol{x},Y,\boldsymbol{\beta
}\right)  \boldsymbol{\Psi}_{\lambda}^{T}\left(  \boldsymbol{x}%
,Y,\boldsymbol{\beta}\right)  \right]  dH(\boldsymbol{x}),
\]
$\mathcal{X}$ is the support of $\boldsymbol{X}$, and
\[
\boldsymbol{J}_{\lambda}\left(  \boldsymbol{\beta}\right)  =E\left[
\frac{\partial\boldsymbol{\Psi}_{\lambda}\left(  \boldsymbol{X}%
,Y,\boldsymbol{\beta}\right)  }{\partial\boldsymbol{\beta}^{T}}\right]  =%
{\displaystyle\int_{\mathcal{X}}}
E\left[  \frac{\partial\boldsymbol{\Psi}_{\lambda}\left(  \boldsymbol{x}%
,Y,\boldsymbol{\beta}\right)  }{\partial\boldsymbol{\beta}^{T}}\right]
dH(\boldsymbol{x}).
\]
In relation to the matrix $\boldsymbol{K}_{\lambda}\left(  \boldsymbol{\beta
}_{0}\right)  $, we have
\[
E\left[  \boldsymbol{\Psi}_{\lambda}\left(  \boldsymbol{x},Y,\boldsymbol{\beta
}\right)  \boldsymbol{\Psi}_{\lambda}^{T}\left(  \boldsymbol{x}%
,Y,\boldsymbol{\beta}\right)  \right]  =\frac{(e^{\lambda\boldsymbol{x}%
^{T}\boldsymbol{\beta}}+e^{\boldsymbol{x}^{T}\boldsymbol{\beta}})^{2}%
}{(1+e^{\boldsymbol{x}^{T}\boldsymbol{\beta}})^{2(\lambda+2)}}E\left[  \left(
e^{\boldsymbol{x}^{T}\boldsymbol{\beta}}-Y(1+e^{\boldsymbol{x}^{T}%
\boldsymbol{\beta}})\right)  ^{2}\right]  \boldsymbol{xx}^{T},
\]
but $E\left[  Y^{2}\right]  =\pi(\boldsymbol{x}^{T}\boldsymbol{\beta})$ and
\[
E\left[  \left(  e^{\boldsymbol{x}^{T}\boldsymbol{\beta}}%
-Y(1+e^{\boldsymbol{x}^{T}\boldsymbol{\beta}})\right)  ^{2}\right]
=e^{\boldsymbol{x}^{T}\boldsymbol{\beta}}.
\]
Therefore
\begin{equation}
\boldsymbol{K}_{\lambda}\left(  \boldsymbol{\beta}\right)  =E\left[
\boldsymbol{\Psi}_{\lambda}\left(  \boldsymbol{X},Y,\boldsymbol{\beta}\right)
\boldsymbol{\Psi}_{\lambda}^{T}\left(  \boldsymbol{X},Y,\boldsymbol{\beta
}\right)  \right]  =%
{\displaystyle\int_{\mathcal{X}}}
\frac{(e^{\lambda\boldsymbol{x}^{T}\boldsymbol{\beta}}+e^{\boldsymbol{x}%
^{T}\boldsymbol{\beta}})^{2}}{(1+e^{\boldsymbol{x}^{T}\boldsymbol{\beta}%
})^{2(\lambda+2)}}e^{\boldsymbol{x}^{T}\boldsymbol{\beta}}\boldsymbol{xx}%
^{T}dH(\boldsymbol{x}).\label{K1}%
\end{equation}
An estimator of $\boldsymbol{K}_{\lambda}\left(  \boldsymbol{\beta}\right)  $
will be%
\[
\widehat{\boldsymbol{K}}_{\lambda}\left(  \boldsymbol{\beta}\right)  =%
{\displaystyle\int_{\mathcal{X}}}
\frac{(e^{\lambda\boldsymbol{x}^{T}\boldsymbol{\beta}}+e^{\boldsymbol{x}%
^{T}\boldsymbol{\beta}})^{2}}{(1+e^{\boldsymbol{x}^{T}\boldsymbol{\beta}%
})^{2(\lambda+2)}}e^{\boldsymbol{x}^{T}\boldsymbol{\beta}}\boldsymbol{xx}%
^{T}dH_{n}(\boldsymbol{x}),
\]
where $H_{n}(\boldsymbol{x})$ the empirical distribution function associated
with the sample $\boldsymbol{x}_{1},...,\boldsymbol{x}_{n}$. Then%
\begin{equation}
\widehat{\boldsymbol{K}}_{\lambda}\left(  \boldsymbol{\beta}\right)  =\frac
{1}{n}\sum\limits_{i=1}^{n}\frac{(e^{\lambda\boldsymbol{x}^{T}%
\boldsymbol{\beta}}+e^{\boldsymbol{x}^{T}\boldsymbol{\beta}})^{2}%
}{(1+e^{\boldsymbol{x}^{T}\boldsymbol{\beta}})^{2(\lambda+2)}}%
e^{\boldsymbol{x}_{i}^{T}\boldsymbol{\beta}}\boldsymbol{x}_{i}\boldsymbol{x}%
_{i}^{T}.\label{K1e}%
\end{equation}
It is interesting to observe that for $\lambda=0$ we get
\begin{align*}
\widehat{\boldsymbol{K}}_{0}\left(  \boldsymbol{\beta}\right)   &  =\frac
{1}{n}\sum\limits_{i=1}^{n}\frac{(1+e^{\boldsymbol{x}_{i}^{T}\boldsymbol{\beta
}})^{2}}{(1+e^{\boldsymbol{x}_{i}^{T}\boldsymbol{\beta}})^{4}}%
e^{\boldsymbol{x}_{i}^{T}\boldsymbol{\beta}}\boldsymbol{x}_{i}\boldsymbol{x}%
_{i}^{T}=\frac{1}{n}\mathbb{X}^{T}\text{ }\mathrm{diag}\left(  \pi
_{i}(\boldsymbol{x}^{T}\boldsymbol{\beta})\left(  1-\pi_{i}(\boldsymbol{x}%
^{T}\boldsymbol{\beta})\right)  \right)  _{i=1,...,n}\mathbb{X}\\
&  =\boldsymbol{I}_{F}\left(  \boldsymbol{\beta}\right)  ,
\end{align*}
with $\boldsymbol{I}_{F}\left(  \boldsymbol{\beta}_{0}\right)  $ being the
Fisher information matrix associated to the logistic regression model.

To compute the matrix $\boldsymbol{J}_{\lambda}\left(  \boldsymbol{\beta}%
_{0}\right)  $, first we need to calculate
\[
\frac{\partial\boldsymbol{\Psi}_{\lambda}\left(  \boldsymbol{x}%
,y,\boldsymbol{\beta}\right)  }{\partial\boldsymbol{\beta}^{T}}=L_{1}%
(\boldsymbol{x},y,\boldsymbol{\beta})+L_{2}(\boldsymbol{x},y,\boldsymbol{\beta
}),
\]
where
\[
L_{1}(\boldsymbol{x},y,\boldsymbol{\beta})=(\lambda e^{\lambda\boldsymbol{x}%
^{T}\boldsymbol{\beta}}+e^{\boldsymbol{x}^{T}\boldsymbol{\beta}}%
)\frac{e^{\boldsymbol{x}^{T}\boldsymbol{\beta}}-y(1+e^{\boldsymbol{x}%
^{T}\boldsymbol{\beta}})}{(1+e^{\boldsymbol{x}^{T}\boldsymbol{\beta}%
})^{\lambda+2}}\boldsymbol{xx}^{T}%
\]
and
\begin{align*}
L_{2}(\boldsymbol{x},y,\boldsymbol{\beta}) &  =(e^{\lambda\boldsymbol{x}%
^{T}\boldsymbol{\beta}}+e^{\boldsymbol{x}^{T}\boldsymbol{\beta}})\left(
\frac{\left(  e^{\boldsymbol{x}^{T}\boldsymbol{\beta}}-ye^{\boldsymbol{x}%
^{T}\boldsymbol{\beta}}\right)  (1+e^{\boldsymbol{x}^{T}\boldsymbol{\beta}%
})^{\lambda+2}}{(1+e^{\boldsymbol{x}^{T}\boldsymbol{\beta}})^{2(\lambda+2)}%
}\right. \\
&  \left.  -\frac{\left(  \lambda+2\right)  \left(  (1+e^{\boldsymbol{x}%
^{T}\boldsymbol{\beta}})^{\lambda+1}\right)  e^{\boldsymbol{x}^{T}%
\boldsymbol{\beta}}\left(  e^{\boldsymbol{x}^{T}\boldsymbol{\beta}%
}-y(1+e^{\boldsymbol{x}^{T}\boldsymbol{\beta}})\right)  }{(1+e^{\boldsymbol{x}%
^{T}\boldsymbol{\beta}})^{2(\lambda+2)}}\right)  \boldsymbol{xx}^{T},
\end{align*}
and hence%
\[
E\left[  \frac{\partial\boldsymbol{\Psi}_{\lambda}\left(  \boldsymbol{x}%
,Y,\boldsymbol{\beta}\right)  }{\partial\boldsymbol{\beta}^{T}}\right]
=E\left[  L_{1}(\boldsymbol{x},Y,\boldsymbol{\beta})\right]  +E\left[
L_{2}(\boldsymbol{x},Y,\boldsymbol{\beta})\right]  .
\]
But
\[
E\left[  e^{\boldsymbol{x}^{T}\boldsymbol{\beta}}-Y(1+e^{\boldsymbol{x}%
^{T}\boldsymbol{\beta}})\right]  =e^{\boldsymbol{x}^{T}\boldsymbol{\beta}%
}-\frac{e^{\boldsymbol{x}^{T}\boldsymbol{\beta}}}{1+e^{\boldsymbol{x}%
^{T}\boldsymbol{\beta}}}(1+e^{\boldsymbol{x}^{T}\boldsymbol{\beta}})=0.
\]
Therefore%
\[
E\left[  L_{1}(\boldsymbol{x},Y,\boldsymbol{\beta})\right]  =\boldsymbol{0}%
_{(k+1)(k+1)}.
\]
On the other hand
\begin{align*}
E\left[  L_{2}(\boldsymbol{x},Y,\boldsymbol{\beta})\right]   &  =\frac
{e^{\lambda\boldsymbol{x}^{T}\boldsymbol{\beta}}+e^{\boldsymbol{x}%
^{T}\boldsymbol{\beta}}}{(1+e^{\boldsymbol{x}^{T}\boldsymbol{\beta}%
})^{2(\lambda+2)}}\left(  (1+e^{\boldsymbol{x}^{T}\boldsymbol{\beta}%
})^{\lambda+2}E\left[  e^{\boldsymbol{x}^{T}\boldsymbol{\beta}}%
-Ye^{\boldsymbol{x}^{T}\boldsymbol{\beta}}\right]  \right. \\
&  \left.  +\left(  \lambda+2\right)  (1+e^{\boldsymbol{x}^{T}%
\boldsymbol{\beta}})^{\lambda+1}e^{\boldsymbol{x}^{T}\boldsymbol{\beta}%
}E\left[  e^{\boldsymbol{x}^{T}\boldsymbol{\beta}}-Y(1+e^{\boldsymbol{x}%
^{T}\boldsymbol{\beta}})\right]  \right)  \boldsymbol{xx}^{T}\\
&  =\frac{e^{\lambda\boldsymbol{x}^{T}\boldsymbol{\beta}}+e^{\boldsymbol{x}%
^{T}\boldsymbol{\beta}}}{(1+e^{\boldsymbol{x}^{T}\boldsymbol{\beta}}%
)^{\lambda+3}}e^{\boldsymbol{x}^{T}\boldsymbol{\beta}}\boldsymbol{xx}^{T}.
\end{align*}
Finally,%
\begin{align}
\boldsymbol{J}_{\lambda}\left(  \boldsymbol{\beta}\right)   &  =\int%
_{\mathcal{X}}E\left[  \frac{\partial\boldsymbol{\Psi}_{\lambda}\left(
\boldsymbol{x},Y,\boldsymbol{\beta}\right)  }{\partial\boldsymbol{\beta}^{T}%
}\right]  dH(\boldsymbol{x)}\label{J1}\\
&  =\int_{\mathcal{X}}\frac{e^{\lambda\boldsymbol{x}^{T}\boldsymbol{\beta}%
}+e^{\boldsymbol{x}^{T}\boldsymbol{\beta}}}{(1+e^{\boldsymbol{x}%
^{T}\boldsymbol{\beta}})^{\lambda+3}}e^{\boldsymbol{x}^{T}\boldsymbol{\beta}%
}\boldsymbol{xx}^{T}dH(\boldsymbol{x}),\nonumber
\end{align}
and an estimator of $\boldsymbol{J}_{\lambda}\left(  \boldsymbol{\beta}%
_{0}\right)  $ is given by
\begin{equation}
\widehat{\boldsymbol{J}}_{\lambda}\left(  \boldsymbol{\beta}\right)  =\frac
{1}{n}\sum\limits_{i=1}^{n}\frac{e^{\lambda\boldsymbol{x}_{i}^{T}%
\boldsymbol{\beta}}+e^{\boldsymbol{x}_{i}^{T}\boldsymbol{\beta}}%
}{(1+e^{\boldsymbol{x}_{i}^{T}\boldsymbol{\beta}})^{\lambda+3}}%
e^{\boldsymbol{x}_{i}^{T}\boldsymbol{\beta}}\boldsymbol{x}_{i}\boldsymbol{x}%
_{i}^{T}.\label{J1e}%
\end{equation}
In particular, for $\lambda=0$, we have%
\begin{align*}
\widehat{\boldsymbol{J}}_{0}\left(  \boldsymbol{\beta}\right)   &  =\frac
{1}{n}\mathbb{X}^{T}\text{ }\mathrm{diag}\left(  \pi_{i}(\boldsymbol{x}%
^{T}\boldsymbol{\beta})\left(  1-\pi_{i}(\boldsymbol{x}^{T}\boldsymbol{\beta
})\right)  \right)  _{i=1,...,n}\mathbb{X}\\
&  =\boldsymbol{I}_{F}\left(  \boldsymbol{\beta}\right)  .
\end{align*}
From the sequence of above results, the next theorem follows.

\begin{theorem}
The asymptotic distribution of the MDPDE for $\boldsymbol{\beta}$,
$\widehat{\boldsymbol{\beta}}_{\lambda}$, is given by
\[
\sqrt{n}(\widehat{\boldsymbol{\beta}}_{\lambda}-\boldsymbol{\beta}%
_{0})\underset{n\rightarrow\infty}{\overset{\mathcal{L}}{\longrightarrow}%
}\mathcal{N}\left(  \boldsymbol{0},\boldsymbol{\Sigma}_{\lambda}\left(
\boldsymbol{\beta}_{0}\right)  \right)
\]
where
\[
\boldsymbol{\Sigma}_{\lambda}\left(  \boldsymbol{\beta}_{0}\right)
=\boldsymbol{J}_{\lambda}^{-1}\left(  \boldsymbol{\beta}_{0}\right)
\boldsymbol{K}_{\lambda}\left(  \boldsymbol{\beta}_{0}\right)  \boldsymbol{J}%
_{\lambda}^{-1}\left(  \boldsymbol{\beta}_{0}\right)
\]
and the matrices $\boldsymbol{J}_{\lambda}\left(  \boldsymbol{\beta}%
_{0}\right)  $ and $\boldsymbol{K}_{\lambda}\left(  \boldsymbol{\beta}%
_{0}\right)  $ where defined in (\ref{J1}) and (\ref{K1}), respectively.
\end{theorem}

\begin{remark}
\label{rem}We have considered that the covariates are random, a crutial
assumption to get the asymptotic distribution of the MDPDE by using,
\textquotedblleft in part\textquotedblright, the standard asymptotic theory
for M-estimators. It is interesting to highlight that whenever the covariates
were non-stochastic (fixed design case), the asymptotic distribution of the
MDPDE could be obtained from Ghosh and Basu (2013) without using the standard
asymptotic theory of M-estimators. In order to present the results in the most
general setting, we shall assume that the random variables $Y_{i}$ with
$i=1,...,I$, are binomial with parameters $n_{i}$ and $\pi_{i}=\pi
(\boldsymbol{x}_{i}^{T}\boldsymbol{\beta})$ instead of Bernoulli random
variables. We shall denote by $N=\sum_{i=1}^{I}n_{i}$ and let $n_{i1}$ denotes
the observed value of $Y_{i}$. We will assume that $I$ is fixed and for each
$i=1,\ldots,I$, construct the independent and identically distributed latent
observations $z_{i1},\ldots,z_{in_{i}}$ each following a Bernoulli
distribution with probability $\pi$ and $n_{i1}=\sum_{j=1}^{n_{i}}z_{ij}$.
Then, $N$ random observations $z_{11},\ldots,z_{1n_{1}}$, $z_{21}%
,\ldots,z_{2n_{2}}$, $\ldots$, $z_{I1},\ldots,z_{In_{I}}$ are independent but
have possibly different distribution with $z_{ij}\sim Ber(\pi_{i})$. This
falls under the general set-up of independent but non-homogeneous observations
as considered in Ghosh and Basu (2013) and hence it is immediately seen that
the corresponding estimating equations for the MDPDE,
$\widehat{\boldsymbol{\beta}}_{\lambda}^{\ast}$ in this context, for
$\lambda>0$ are given by
\[
\sum_{i=1}^{I}\frac{e^{\lambda\boldsymbol{x}_{i}^{T}\boldsymbol{\beta}%
}+e^{\boldsymbol{x}_{i}^{T}\boldsymbol{\beta}}}{(1+e^{\boldsymbol{x}_{i}%
^{T}\boldsymbol{\beta}})^{\lambda+1}}\left(  n_{i}\pi(\boldsymbol{x}_{i}%
^{T}\boldsymbol{\beta})-n_{i1}\right)  \boldsymbol{x}_{i}=\boldsymbol{0}%
\]
and for $\lambda=0$, by
\begin{equation}
\sum_{i=1}^{I}\left(  n_{i}\pi(\boldsymbol{x}_{i}^{T}\boldsymbol{\beta
})-n_{i1}\right)  \boldsymbol{x}_{i}=\boldsymbol{0}.\label{Vero}%
\end{equation}
Now, assuming
\[
\lim_{N\rightarrow\infty}\frac{n_{i}}{N}=\alpha_{i}\in\left(  0,1\right)
,\text{ }i=1,...,I,
\]
and following Ghosh and Basu (2013), we get the asymptotic distribution of the
MDPDE of $\boldsymbol{\beta}$, $\widehat{\boldsymbol{\beta}}_{\lambda}^{\ast}%
$, as given by
\begin{equation}
\sqrt{N}(\widehat{\boldsymbol{\beta}}_{\lambda}^{\ast}-\boldsymbol{\beta}%
_{0})\underset{N\rightarrow\infty}{\overset{\mathcal{L}}{\longrightarrow}%
}\mathcal{N}\left(  \boldsymbol{0},\boldsymbol{\Sigma}^{\ast}\left(
\boldsymbol{\beta}_{0}\right)  \right) \label{3.1}%
\end{equation}
where
\[
\boldsymbol{\Sigma}^{\ast}\left(  \boldsymbol{\beta}_{0}\right)
=\boldsymbol{J}^{\ast-1}\left(  \boldsymbol{\beta}_{0}\right)  \boldsymbol{K}%
^{\ast}\left(  \boldsymbol{\beta}_{0}\right)  \boldsymbol{J}^{\ast-1}\left(
\boldsymbol{\beta}_{0}\right)  .
\]
Here, the matrices $\boldsymbol{J}^{\ast}\left(  \boldsymbol{\beta}%
_{0}\right)  $ and $\boldsymbol{K}^{\ast}\left(  \boldsymbol{\beta}%
_{0}\right)  $ can be obtained directly from the general results of Ghosh and
Basu (2013) or from the simplified results in the context of Bernoulli
logistic regression with fixed design in Ghosh and Basu (2015) and are given
by
\[
\boldsymbol{J}^{\ast}\left(  \boldsymbol{\beta}_{0}\right)  =\sum_{i=1}%
^{I}\alpha_{i}e^{\boldsymbol{x}_{i}^{T}\boldsymbol{\beta}}\frac{e^{\lambda
\boldsymbol{x}_{i}^{T}\boldsymbol{\beta}}+e^{\boldsymbol{x}_{i}^{T}%
\boldsymbol{\beta}}}{(1+e^{\boldsymbol{x}_{i}^{T}\boldsymbol{\beta}}%
)^{\lambda+3}}\boldsymbol{x}_{i}\boldsymbol{x}_{i}^{T},
\]
and
\[
\boldsymbol{K}^{\ast}\left(  \boldsymbol{\beta}_{0}\right)  =\sum_{i=1}%
^{I}\alpha_{i}e^{\boldsymbol{x}_{i}^{T}\boldsymbol{\beta}}\frac{(e^{\lambda
\boldsymbol{x}_{i}^{T}\boldsymbol{\beta}}+e^{\boldsymbol{x}_{i}^{T}%
\boldsymbol{\beta}})^{2}}{(1+e^{\boldsymbol{x}_{i}^{T}\boldsymbol{\beta}%
})^{2(\lambda+2)}}\boldsymbol{x}_{i}\boldsymbol{x}_{i}^{T}.
\]
For $\lambda=0$, it is clear, based on (\ref{Vero}), that we get the classical
likelihood estimator. We can observe that in this situation
\[
\boldsymbol{J}^{\ast}\left(  \boldsymbol{\beta}_{0}\right)  =\boldsymbol{K}%
^{\ast}\left(  \boldsymbol{\beta}_{0}\right)  =\boldsymbol{I}_{F}\left(
\boldsymbol{\beta}_{0}\right)
\]
and we get the classical result,
\[
\sqrt{N}(\widehat{\boldsymbol{\beta}}_{\lambda=0}^{\ast}-\boldsymbol{\beta
}_{0})\underset{N\rightarrow\infty}{\overset{\mathcal{L}}{\longrightarrow}%
}\mathcal{N}\left(  \boldsymbol{0},\boldsymbol{I}_{F}^{-1}\left(
\boldsymbol{\beta}_{0}\right)  \right)  .
\]

\end{remark}

\section{Wald type test statistic for testing linear hypothesis\label{sec3}}

Based on the asymptotic distribution of $\widehat{\boldsymbol{\beta}}%
_{\lambda}$ we are going to define a family of Wald-type test statistics for
testing the null hypothesis
\begin{equation}
H_{0}:\boldsymbol{M}^{T}\boldsymbol{\beta=m},\label{3.10}%
\end{equation}
where $\boldsymbol{M}^{T}$ is any matrix of $r$ rows and $k+1$ columns and $m
$ a vector of order $r$ of specified constant. We assume that the matrix
$\boldsymbol{M}^{T}$ has full row rank, i.e., $\mathrm{rank}\left(
\boldsymbol{M}\right)  =r$.

\begin{definition}
Let $\widehat{\boldsymbol{\beta}}_{\lambda}$ be the minimum power divergence
estimator. The family of Wald type test statistics for testing the null
hypothesis given in (\ref{3.10}) is given by
\begin{align}
W_{n} &  =n(\boldsymbol{M}^{T}\widehat{\boldsymbol{\beta}}_{\lambda
}-\boldsymbol{m})^{T}\left(  \boldsymbol{M}^{T}\boldsymbol{J}_{\lambda}%
^{-1}(\widehat{\boldsymbol{\beta}}_{\lambda})\boldsymbol{K}_{\lambda
}(\widehat{\boldsymbol{\beta}}_{\lambda})\boldsymbol{J}_{\lambda}%
^{-1}(\widehat{\boldsymbol{\beta}}_{\lambda})\boldsymbol{M}\right)
^{-1}(\boldsymbol{M}^{T}\widehat{\boldsymbol{\beta}}_{\lambda}-\boldsymbol{m}%
)\nonumber\\
&  =n(\boldsymbol{M}^{T}\widehat{\boldsymbol{\beta}}_{\lambda}-\boldsymbol{m}%
)^{T}(\boldsymbol{M}^{T}\boldsymbol{\Sigma}_{\lambda}%
(\widehat{\boldsymbol{\beta}}_{\lambda})\boldsymbol{M})^{-1}(\boldsymbol{M}%
^{T}\widehat{\boldsymbol{\beta}}_{\lambda}-\boldsymbol{m}).\label{3.11}%
\end{align}

\end{definition}

In the particular case of $\lambda=0$, i.e. $\widehat{\boldsymbol{\beta}}$ is
the MLE, we get the classical Wald test statistic because in this case
\[
\boldsymbol{J}_{\lambda=0}^{-1}\left(  \boldsymbol{\beta}_{0}\right)
\boldsymbol{K}_{\lambda=0}\left(  \boldsymbol{\beta}_{0}\right)
\boldsymbol{J}_{\lambda=0}^{-1}\left(  \boldsymbol{\beta}_{0}\right)
\boldsymbol{=I}_{F}^{-1}\left(  \boldsymbol{\beta}_{0}\right)  .
\]

\begin{theorem}
The asymptotic distribution of the Wald type test statistic, $W_{n}$, defined
in (\ref{3.11}), under the null hypothesis given in (\ref{3.10}), is a
chi-square distribution with $r$ degrees of freedom.
\end{theorem}

\begin{proof}
We have $\boldsymbol{M}^{T}\widehat{\boldsymbol{\beta}}_{\lambda
}-\boldsymbol{m=}\boldsymbol{M}^{T}(\widehat{\boldsymbol{\beta}}_{\lambda
}-\boldsymbol{\beta}_{0})$ and $\sqrt{n}(\widehat{\boldsymbol{\beta}}%
_{\lambda}-\boldsymbol{\beta}_{0})\underset{n\rightarrow\infty
}{\overset{\mathcal{L}}{\longrightarrow}}\mathcal{N}\left(  \boldsymbol{0}%
,\boldsymbol{\Sigma}_{\lambda}\left(  \boldsymbol{\beta}_{0}\right)  \right)
$. Therefore
\[
\sqrt{n}(\boldsymbol{M}^{T}\widehat{\boldsymbol{\beta}}_{\lambda
}-\boldsymbol{m})\underset{n\rightarrow\infty}{\overset{\mathcal{L}%
}{\longrightarrow}}\mathcal{N}\left(  \boldsymbol{0}%
,\boldsymbol{\boldsymbol{M}}^{T}\boldsymbol{\Sigma}_{\lambda}\left(
\boldsymbol{\beta}_{0}\right)  \boldsymbol{M}\right)
\]
and since $\boldsymbol{\boldsymbol{M}^{T}\Sigma}_{\lambda}\left(
\boldsymbol{\beta}_{0}\right)  \boldsymbol{M}\left(  \boldsymbol{M}%
^{T}\boldsymbol{J}_{\lambda}^{-1}\left(  \boldsymbol{\beta}_{0}\right)
\boldsymbol{K}_{\lambda}\left(  \boldsymbol{\beta}_{0}\right)  \boldsymbol{J}%
_{\lambda}^{-1}\left(  \boldsymbol{\beta}_{0}\right)  \boldsymbol{M}\right)
^{-1}=\boldsymbol{I}_{r\times r}$, the asymptotic distribution of $W_{n}$ is a
chi-square distribution with $r$ degrees of freedom.
\end{proof}

\begin{remark}
If we consider
\begin{equation}
\boldsymbol{M}^{T}=\left(
\begin{array}
[c]{cc}%
\boldsymbol{0}_{k\times1} & \boldsymbol{I}_{k\times k}%
\end{array}
\right)  _{k\times(k+1)}\label{3.12}%
\end{equation}
we have
\[
\boldsymbol{M}^{T}\boldsymbol{\beta}=\boldsymbol{0},
\]
if and only if $\beta_{i}=0$, $i=1,...,k$. Therefore, we can consider the
Wald-type test statistics with $\boldsymbol{M}^{T}$ defined in (\ref{3.12})
for testing%
\[
H_{0}:\beta_{1}=\beta_{2}=\cdots=\beta_{k}=0.
\]
In this case, the asymptotic distribution of the Wald type test statistic is a
chi square distribution with $k$ degrees of freedom. If we consider
$\boldsymbol{M}^{T}$ to be a vector with all elements equal zero except for
the ($i+1$)-th term, equals $1$, we can test
\[
H_{0}:\beta_{i}=0.
\]

\end{remark}

Based on the previous theorem the null hypothesis given in (\ref{3.10}) will
be rejected if we have that
\begin{equation}
W_{n}>\chi_{r,\alpha}^{2},\label{3.111}%
\end{equation}
where $\chi_{r,\alpha}^{2}$ is the quantile of order $1-\alpha$.for a
chi-square with $r$ degrees of freedom Let us consider $\boldsymbol{\beta
}^{\ast}\in\Theta$ such that $\boldsymbol{M}^{T}\boldsymbol{\beta}^{\ast}%
\neq\boldsymbol{m}$, i.e., $\boldsymbol{\beta}^{\ast}$ does not belong to the
null hypothesis. We denote%
\[
q_{\boldsymbol{\beta}_{1}}(\boldsymbol{\beta}_{2})=\left(  \boldsymbol{M}%
^{T}\boldsymbol{\beta}_{1}-\boldsymbol{m}\right)  ^{T}\left(  \boldsymbol{M}%
^{T}\boldsymbol{\Sigma}_{\lambda}\left(  \boldsymbol{\beta}_{2}\right)
\boldsymbol{M}\right)  ^{-1}\left(  \boldsymbol{M}^{T}\boldsymbol{\beta}%
_{1}-\boldsymbol{m}\right)
\]
and we are going to get an approximation to the power function for the test
statistics given in (\ref{3.111}).

\begin{theorem}
Let $\boldsymbol{\beta}^{\ast}\in\Theta$ ,with $\boldsymbol{M}^{T}%
\boldsymbol{\beta}^{\ast}\neq\boldsymbol{m}$, be the true value of the
parameter so that $\widehat{\boldsymbol{\beta}}_{\lambda}%
\underset{n\rightarrow\infty}{\overset{P}{\longrightarrow}}\boldsymbol{\beta
}^{\ast}$. The power function of the test statistic given in (\ref{3.111}), in
$\boldsymbol{\beta}^{\ast}$, is given by
\begin{equation}
\pi\left(  \boldsymbol{\beta}^{\ast}\right)  =1-\Phi_{n}\left(  \frac
{1}{\sigma\left(  \boldsymbol{\beta}^{\ast}\right)  }\left(  \frac
{\chi_{r,\alpha}^{2}}{\sqrt{n}}-\sqrt{n}q_{\boldsymbol{\beta}^{\ast}%
}(\boldsymbol{\beta}^{\ast})\right)  \right)  ,\label{3.112}%
\end{equation}
where $\Phi_{n}\left(  x\right)  $ tends uniformly to the standard normal
distribution $\Phi\left(  x\right)  $ and $\sigma\left(  \boldsymbol{\beta
}^{\ast}\right)  $ is given by%
\[
\sigma^{2}\left(  \boldsymbol{\beta}^{\ast}\right)  =\left.  \frac{\partial
q_{\boldsymbol{\beta}}(\boldsymbol{\beta}^{\ast})}{\partial\boldsymbol{\beta
}^{T}}\right\vert _{\boldsymbol{\beta}=\boldsymbol{\beta}^{\ast}%
}\boldsymbol{\Sigma}_{\lambda}\left(  \boldsymbol{\beta}_{0}\right)  \left.
\frac{\partial q_{\boldsymbol{\beta}}(\boldsymbol{\beta}^{\ast})}%
{\partial\boldsymbol{\beta}}\right\vert _{\boldsymbol{\beta}=\boldsymbol{\beta
}^{\ast}}.
\]

\end{theorem}

\begin{proof}
We have
\begin{align*}
\pi\left(  \boldsymbol{\beta}^{\ast}\right)   &  =\Pr\left(  W_{n}%
>\chi_{r,\alpha}^{2}\right)  =\Pr\left(  n\left(
q_{\widehat{\boldsymbol{\beta}}_{\lambda}}(\widehat{\boldsymbol{\beta}%
}_{\lambda})-q_{\boldsymbol{\beta}^{\ast}}(\boldsymbol{\beta}^{\ast})\right)
>\chi_{r,\alpha}^{2}-nq_{\boldsymbol{\beta}^{\ast}}(\boldsymbol{\beta}^{\ast
})\right) \\
&  =\Pr\left(  \sqrt{n}\left(  q_{\widehat{\boldsymbol{\beta}}_{\lambda}%
}(\widehat{\boldsymbol{\beta}}_{\lambda})-q_{\boldsymbol{\beta}^{\ast}%
}(\boldsymbol{\beta}^{\ast})\right)  >\frac{\chi_{r,\alpha}^{2}}{\sqrt{n}%
}-\sqrt{n}q_{\boldsymbol{\beta}^{\ast}}(\boldsymbol{\beta}^{\ast})\right)  .
\end{align*}
Now we are going to get the asymptotic distribution of the random variable
$\sqrt{n}(q_{\widehat{\boldsymbol{\beta}}_{\lambda}}%
(\widehat{\boldsymbol{\beta}}_{\lambda})-q_{\boldsymbol{\beta}^{\ast}%
}(\boldsymbol{\beta}^{\ast}))$. It is clear that
$q_{\widehat{\boldsymbol{\beta}}_{\lambda}}(\widehat{\boldsymbol{\beta}%
}_{\lambda})$ and $q_{\widehat{\boldsymbol{\beta}}_{\lambda}}%
(\boldsymbol{\beta}^{\ast})$ have the same asymptotic distribution because
$\widehat{\boldsymbol{\beta}}_{\lambda}\underset{n\rightarrow\infty
}{\overset{P}{\longrightarrow}}\boldsymbol{\beta}^{\ast}$. A first order
Taylor expansion of $q_{\widehat{\boldsymbol{\beta}}_{\lambda}}%
(\boldsymbol{\beta}^{\ast})$ at $\widehat{\boldsymbol{\beta}}_{\lambda}$
around $\boldsymbol{\beta}^{\ast}$ gives%
\[
q_{\widehat{\boldsymbol{\beta}}_{\lambda}}(\boldsymbol{\beta}^{\ast
})-q_{\boldsymbol{\beta}^{\ast}}(\boldsymbol{\beta}^{\ast})=\left.
\frac{\partial q_{\boldsymbol{\beta}}(\boldsymbol{\beta}^{\ast})}%
{\partial\boldsymbol{\beta}^{T}}\right\vert _{\boldsymbol{\beta}%
=\boldsymbol{\beta}^{\ast}}(\widehat{\boldsymbol{\beta}}_{\lambda
}-\boldsymbol{\beta}^{\ast})+o_{p}\left(  \left\Vert
\widehat{\boldsymbol{\beta}}_{\lambda}-\boldsymbol{\beta}^{\ast}\right\Vert
\right)  .
\]
Therefore it holds%
\[
\sqrt{n}\left(  q_{\widehat{\boldsymbol{\beta}}_{\lambda}}%
(\widehat{\boldsymbol{\beta}}_{\lambda})-q_{\boldsymbol{\beta}^{\ast}%
}(\boldsymbol{\beta}^{\ast})\right)  \underset{n\rightarrow\infty
}{\overset{\mathcal{L}}{\longrightarrow}}\mathcal{N}\left(  0,\sigma
^{2}\left(  \boldsymbol{\beta}^{\ast}\right)  \right)  ,
\]
where
\[
\sigma^{2}\left(  \boldsymbol{\beta}^{\ast}\right)  =\left.  \frac{\partial
q_{\boldsymbol{\beta}}(\boldsymbol{\beta}^{\ast})}{\partial\boldsymbol{\beta
}^{T}}\right\vert _{\boldsymbol{\beta}=\boldsymbol{\beta}^{\ast}%
}\boldsymbol{J}_{\lambda}^{-1}\left(  \boldsymbol{\beta}_{0}\right)
\boldsymbol{K}_{\lambda}\left(  \boldsymbol{\beta}_{0}\right)  \boldsymbol{J}%
_{\lambda}^{-1}\left(  \boldsymbol{\beta}_{0}\right)  \left.  \frac{\partial
q_{\boldsymbol{\beta}}(\boldsymbol{\beta}^{\ast})}{\partial\boldsymbol{\beta}%
}\right\vert _{\boldsymbol{\beta}=\boldsymbol{\beta}^{\ast}}.
\]
Now the result follows.
\end{proof}

\begin{remark}
Based on the previous theorem we can obtain the sample size necessary to get a
fix power $\pi\left(  \boldsymbol{\beta}^{\ast}\right)  =\pi_{0}$. From
(\ref{3.112}), we must solve the equation
\[
1-\pi_{0}=\Phi\left(  \frac{1}{\sigma\left(  \boldsymbol{\beta}^{\ast}\right)
}\left(  \frac{\chi_{r,\alpha}^{2}}{\sqrt{n}}-\sqrt{n}q_{\boldsymbol{\beta
}^{\ast}}(\boldsymbol{\beta}^{\ast})\right)  \right)
\]
and we get that $n=\left[  n^{\ast}\right]  +1$ with
\[
n^{\ast}=\frac{A+B+\sqrt{A(A+2B)}}{2q_{\boldsymbol{\beta}^{\ast}}%
^{2}(\boldsymbol{\beta}^{\ast})}%
\]
being
\[
A=\sigma^{2}\left(  \boldsymbol{\beta}^{\ast}\right)  \left(  \Phi^{-1}\left(
1-\pi_{0}\right)  \right)  ^{2}\text{ and }B=2q_{\boldsymbol{\beta}^{\ast}%
}(\boldsymbol{\beta}^{\ast})\chi_{r,\alpha}^{2}.
\]

\end{remark}

In the following theorem we present an approximation to the power function at
the contiguous alternative hypothesis
\begin{equation}
\boldsymbol{\beta}_{n}=\boldsymbol{\beta}_{0}+n^{-1/2}\boldsymbol{d}%
,\label{3.2}%
\end{equation}
with $\boldsymbol{d}$ satisfying $\boldsymbol{\beta}_{0}+n^{-1/2}%
\boldsymbol{d}\in\Theta$.

\begin{theorem}
An approximation of the power function for the test statistic given in
(\ref{3.111}), in $\boldsymbol{\beta}_{n}=\boldsymbol{\beta}_{0}%
+n^{-1/2}\boldsymbol{d}$ is given by
\[
\pi\left(  \boldsymbol{\beta}_{n}\right)  =1-F_{\chi_{r}^{2}\left(
\delta\right)  }\left(  \chi_{r,\alpha}^{2}\right)  ,
\]
where $F_{\chi_{r}^{2}(\delta)}$ is the distribution function of a non-central
chi-square with $p$ degrees of freedom and non-centrality parameter $\delta$
given by $\delta=\boldsymbol{d}^{T}\boldsymbol{\Sigma}_{\lambda}\left(
\boldsymbol{\beta}_{0}\right)  \boldsymbol{d}$.
\end{theorem}

\section{Robustness Analysis\label{sec4}}

\subsection{Influence function of the MDPDE}

We will consider the influence function analysis of Hampel et al.~(1986) to
study the robustness of our proposed MDPDE and the corresponding Wald-type
test of general linear hypothesis in the logistic regression model. Since the
MDPDE can be written in term of a $M$-estimator as shown in Section \ref{sec2}
with $\psi$-function given by (\ref{1.100}), we can apply directly the results
of the M-estimation theory of Hampel et al.~(1986) in order to get the
influence function of the proposed MDPDE.

However, we first need to re-define the minimum density power divergence
estimator $\widehat{\boldsymbol{\beta}}_{\lambda}$ from Definition 1 in terms
of a statistical functional. Let us assume the stochastic nature of the
covariates $X$ and that the observations $(X_{1},Y_{1}),\ldots,(X_{n},Y_{n})$
are i.i.d. with some joint distribution $G$. Then we define the required
statistical functional corresponding to $\widehat{\boldsymbol{\beta}}%
_{\lambda}$ as follows.

\begin{definition}
The minimum DPD functional $T_{\lambda}(G)$, corresponding to the minimum DPD
estimator $\widehat{\boldsymbol{\beta}}_{\lambda}$, at the joint distribution
$G$ is defined as the solution of the system of equations
\[
E_{G}\left[  \boldsymbol{\Psi}_{\lambda}(\boldsymbol{X},Y,\boldsymbol{\beta
})\right]  =\boldsymbol{0}%
\]
with respect to $\boldsymbol{\beta}$, whenever the solution exists.
\label{DEF:MDPDE_func}
\end{definition}

Now, if $G_{0}$ denotes the joint model distribution with the true parameter
value $\boldsymbol{\beta}_{0}$ under which
\[
P_{G_{0}}(Y_{i}=1|\boldsymbol{X}_{i}=\boldsymbol{x}_{i})=\pi(\boldsymbol{x}%
_{i}^{T}\boldsymbol{\beta}_{0}),
\]
then it is easy to see that $E_{G_{0}}\left[  \Psi_{\lambda}(\boldsymbol{X}%
,Y,\boldsymbol{\beta}_{0})\right]  =0$ and hence $\boldsymbol{T}_{\lambda
}(G_{0})=$$\boldsymbol{\beta}_{0}$. Therefore, the minimum DPD functional
$\boldsymbol{T}_{\lambda}$ is Fisher consistent.

Next, we can easily obtain the influence function for our MDPDE at the model
distribution $G_{0}$ as presented in the following theorem. This can be
derived either through a straightforward calculation or by applying the
corresponding results from M-estimation theory of Hampel et al.~(1986) and
hence the proof of the theorem is omitted.

\begin{theorem}
The influence function of the minimum DPD functional $T_{\lambda}$, as defined
in Definition \ref{DEF:MDPDE_func} with tuning parameter $\lambda$, at the
model distribution $G_{0}$ is given by
\begin{align*}
\mathcal{IF}((\boldsymbol{x}_{t},y_{t}),T_{\lambda},G_{0}) &  =\boldsymbol{J}%
_{\lambda}^{-1}(\boldsymbol{\beta}_{0})\left(  \boldsymbol{\Psi}_{\lambda
}(\boldsymbol{x}_{t},y_{t},\boldsymbol{\beta}_{0})-E_{G_{0}}[\boldsymbol{\Psi
}_{\lambda}(\boldsymbol{X},Y,\boldsymbol{\beta}_{0})]\right) \\
&  =\boldsymbol{J}_{\lambda}^{-1}(\boldsymbol{\beta}_{0})\boldsymbol{\Psi
}_{\lambda}(\boldsymbol{x}_{t},y_{t},\boldsymbol{\beta}_{0}),
\end{align*}
where $\boldsymbol{J}_{\lambda}(\boldsymbol{\beta})$ is as defined in Section
2 of the paper and $(\boldsymbol{x}_{t},y_{t})$ is the point of contamination.
\label{THM:IF_MDPDE}
\end{theorem}

Before studying the above influence function, let us first recall different
types of outliers in logistic regression model following the discussion in
Croux and Haesbroeck (2003). A contamination point $(x_{t},y_{t})$ will be a
leverage point if $x_{t}$ is outlying in the covariates space and will be a
vertical outlier (in response) if it is not a leverage point but the residual
$y_{t}-\pi(\boldsymbol{x}_{t}^{T}\boldsymbol{\beta})$ is large. Croux and
Haesbroeck (2003) also noted that, for the maximum likelihood estimator of
$\boldsymbol{\beta}$, a vertical outlier or a \textquotedblleft good" leverage
point (for which the residual is small) has bounded influence whereas a bad
leverage point (e.g., misclassified observation etc.) has infinite influence
for $||\boldsymbol{x}_{t}||\rightarrow\infty$.

Next, in order to study the similar nature of the influence function of the
MDPDE having different $\lambda$, note that the influence function given in
Theorem \ref{THM:IF_MDPDE} can be factored into two components as
\[
\mathcal{IF}((\boldsymbol{x}_{t},y_{t}),T_{\lambda},G_{0})=\widetilde{\Psi
}_{\lambda}(\boldsymbol{x}_{t}^{T}\boldsymbol{\beta}_{0},y_{t})\boldsymbol{J}%
_{\lambda}^{-1}(\boldsymbol{\beta}_{0})\boldsymbol{x}_{t},
\]
where the first part $\widetilde{\boldsymbol{\Psi}}_{\lambda}$ depends on the
score, $s=\boldsymbol{x}_{t}^{T}\boldsymbol{\beta}_{0}$, and the response,
$y_{t}$, and is defined as
\[
\widetilde{\Psi}_{\lambda}(s,y)=\frac{\left(  e^{\lambda s}+e^{s}\right)
\left(  e^{s}-y(1+e^{s})\right)  }{\left(  1+e^{s}\right)  ^{\lambda+2}}.
\]
Figure \ref{FIG:psi} shows the nature of this function over the score input at
$y=0,1$ for different values of $\lambda$. Clearly, the function
$\widetilde{\Psi}_{\lambda}$ corresponding to $\lambda=0$ (MLE) is unbounded
as $s\rightarrow\infty$, illustrating the well-known non-robust nature of the
MLE. However, for $\lambda>0$ the function $\widetilde{\Psi}_{\lambda}$ is
bounded in $s$ and becomes more re-descending as $\lambda$ increase, which
implies the increasing robustness of our proposed MDPDEs with increasing
$\lambda>0$.%

\begin{figure}[htbp]  \centering
$\
\begin{tabular}
[c]{cc}%
{\includegraphics[
trim=0.000000in 0.000000in 0.251411in 0.000000in,
height=2.6515in,
width=3.378in
]%
{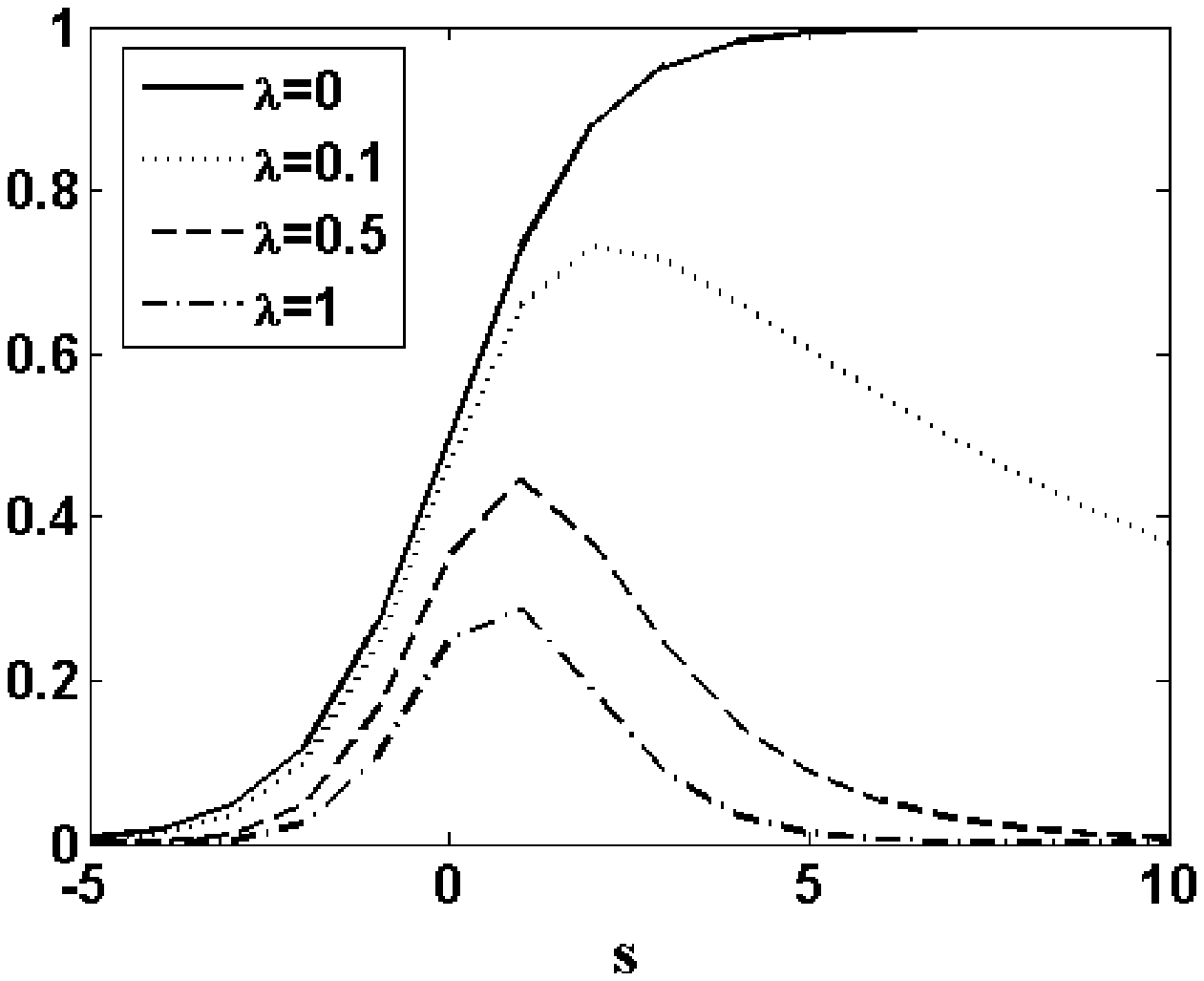}%
}
&
{\includegraphics[
trim=0.250828in 0.000000in 0.000000in 0.000000in,
height=2.6515in,
width=3.378in
]%
{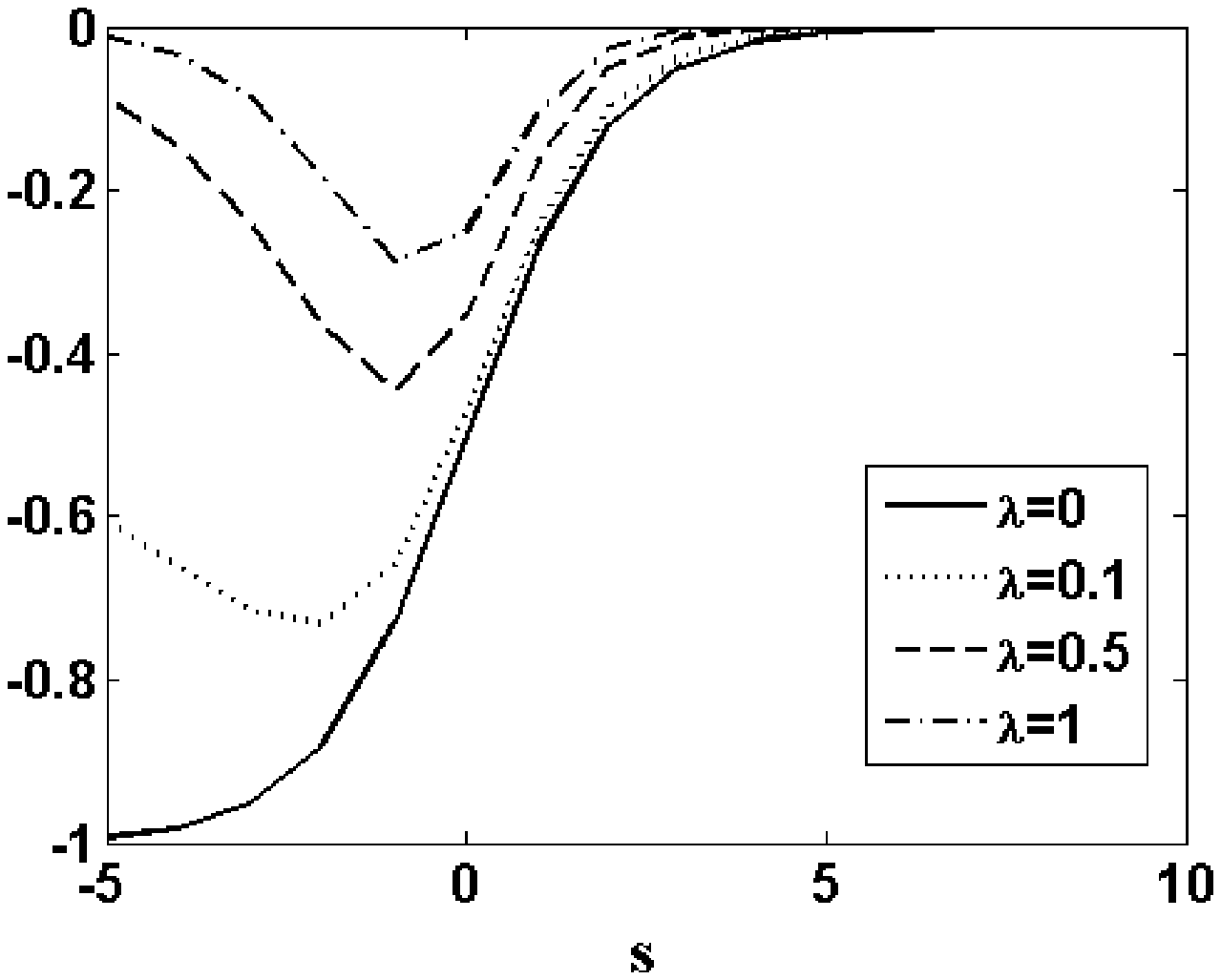}%
}
\\
(a) $y_{t}=0$ & (b) $y_{t}=1$%
\end{tabular}
\ \ $\caption{Plots of $\protect \widetilde{\Psi }_{\protect \lambda }(s;y)$
over $s $ for different $\protect \lambda $ and $y=0,1$.\label{FIG:psi}}%
\end{figure}%

Further, to examine the effect of different types of leverage points more
clearly, following Croux and Haesbroeck (2003), in Figure \ref{FIG:IF_MDPDE},
we present the influence function of the MDPDE of the first slope parameter
$\boldsymbol{\beta}_{1}$ over the covariates values in a logistic regression
model with two independent standard normal covariates and $\boldsymbol{\beta
}_{0}=(0,1,1)^{T}$ fixing $y_{t}=0$ (without loss of generality). We can see
that when both covariates tends to $-\infty$ the influence function becomes
zero for all MDPDEs including the MLE (at $\lambda=0$). These are the
\textquotedblleft good\textquotedblright\ leverage points, as noted in Croux
and Haesbroeck (2003), and all MDPDEs are robust with respect to such good
leverages as in the case of MLE. However, when the covariates approaches to
$\infty$ they yield bad leverage points (generally corresponding to
misclassified points) and have large influence for the MLE ($\lambda=0$). But
the influence function of the MDPDEs with $\lambda>0$ are quite small even for
these bad leverages and become even smaller as $\lambda$ increases. This again
proves the greater robustness of our proposed MDPDEs with larger positive
$\lambda$.%

\begin{figure}[htbp]  \centering
$\
\begin{tabular}
[c]{cc}%
{\includegraphics[
trim=0.000000in 0.000000in 0.418240in 0.000000in,
height=2.6515in,
width=3.2785in
]%
{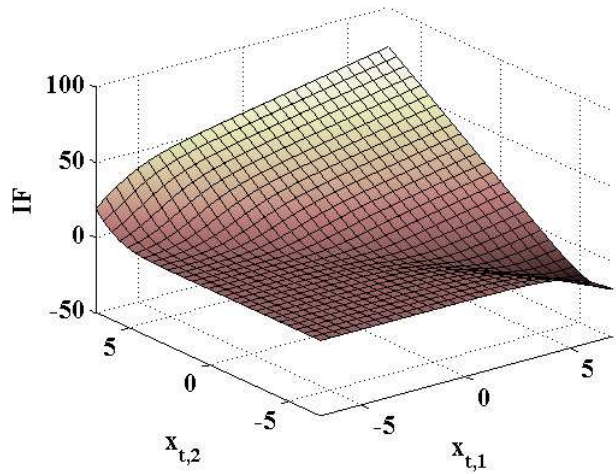}%
}
&
{\includegraphics[
trim=0.000000in 0.000000in 0.418824in 0.000000in,
height=2.6515in,
width=3.2768in
]%
{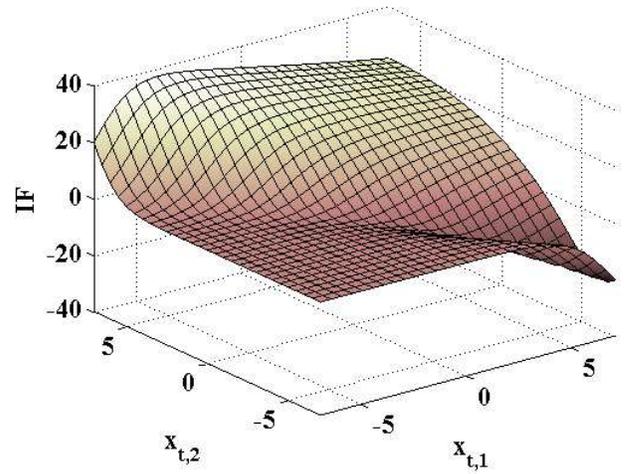}%
}
\\
(a) $\lambda=0$ & (b) $\lambda=0.1$\\%
{\includegraphics[
trim=0.000000in 0.000000in 0.418240in 0.000000in,
height=2.6515in,
width=3.2785in
]%
{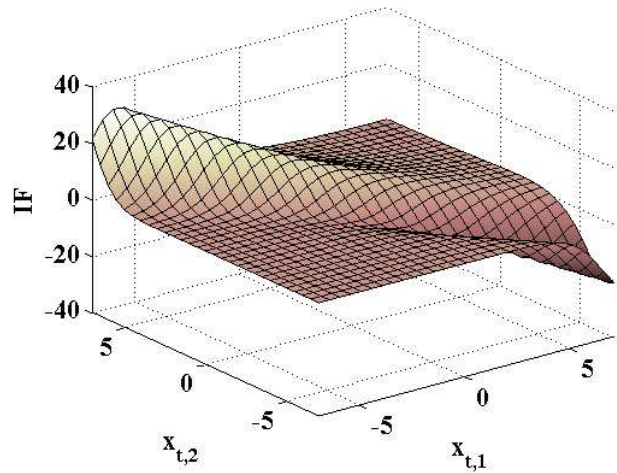}%
}
&
{\includegraphics[
trim=0.000000in 0.000000in 0.418824in 0.000000in,
height=2.6515in,
width=3.2768in
]%
{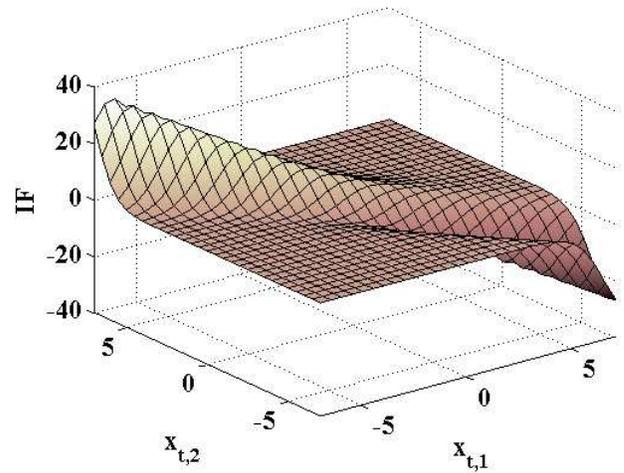}%
}
\\
(c) $\lambda=0.5$ & (d) $\lambda=1$%
\end{tabular}
\ \ $\caption{Influence function of the MDPDE of the first slope parameter
$\protect \beta_1$ for different $\protect \lambda$
($y_t=0$).\label{FIG:IF_MDPDE}}%
\end{figure}%

\begin{remark}
Under the set-up of Remark \ref{rem} with non-stochastic covariate also, we
can derive the influence function of the corresponding MDPDE,
$\widehat{\boldsymbol{\beta}}_{\lambda}^{\ast}$, following Ghosh and Basu
(2013). Whenever the covariates $\boldsymbol{x}_{i}$s are fixed, the
contamination need to be considered over the conditional distribution of
response given covariates which are not identical for each group with given
fixed covariates. Hence, as in Ghosh and Basu (2013), we can consider the
contamination in any one group or in all the group. Following the results in
Ghosh and Basu (2013) or by direct calculation, we get the influence function
of $\widehat{\boldsymbol{\beta}}_{\lambda}^{\ast}$ under contamination only in
one group ($i_{0}$-th, say) with covariate $\boldsymbol{x}_{i_{0}}$ as given
by
\[
\mathcal{IF}_{i_{0}}(y_{t_{i_{0}}},T_{\lambda},G_{0})=\boldsymbol{J}_{\lambda
}^{\ast-1}(\boldsymbol{\beta}_{0})\boldsymbol{\Psi}_{\lambda}(\boldsymbol{x}%
_{i_{0}},y_{t_{i_{0}}},\boldsymbol{\beta}_{0}),
\]
where $y_{t_{i_{0}}}$ is the contamination point in the contaminated
distribution of $Y$ given $\boldsymbol{X}=\boldsymbol{x}_{i_{0}}$. Similarly,
if there is contamination in all the groups with covariates $\boldsymbol{x}%
_{1},\ldots,\boldsymbol{x}_{I}$ respectively at the contamination points
$y_{t_{1}},\ldots,y_{t_{I}}$, then the resulting influence function has the
form
\[
\mathcal{IF}((y_{t_{1}},\ldots,y_{t_{I}}),T_{\lambda},G_{0})=\boldsymbol{J}%
_{\lambda}^{\ast-1}(\boldsymbol{\beta}_{0})\sum_{i=1}^{I}\boldsymbol{\Psi
}_{\lambda}(\boldsymbol{x}_{i},y_{t_{i}},\boldsymbol{\beta}_{0}),
\]
Note that, since the response in a logistic regression takes only values $0$
and $1$, the $y_{t_{i}}$ contamination points all take values only in
$\{0,1\}$ (misclassification errors) and hence all the above influence
functions are bounded with respect to contamination in response for all
$\lambda\geq0$. Hence, the effect of these (misclassification) error in
response cannot be clearly inferred only from these influence functions; see
Pregibon (1982), Copas (1988) and Victoria-Feser (2000) for more such analysis
of misclassification error in logistic regression with fixed design. However,
the above influence functions are bounded in the values of given fixed
covariates only for $\lambda>0$, implying the robustness of the MDPDEs with
$\lambda>0$ and non-robust nature of MLE (at $\lambda=0$) with respect to the
extreme values of the fixed design in any one group.
\end{remark}


\subsection{Influence function of the Wald-Type Test Statistics}

We will now study the robustness of the proposed Wald-type test of Section 3
through the influence function of the corresponding test statistics $W_{n}$
defined in Definition 5. Ignoring the multiplier $n$, let us define the
associated statistical functional for the test statistics $W_{n}$ evaluated at
any joint distribution $G$ as given by
\begin{equation}
W_{\lambda}(G)=\left(  \boldsymbol{M}^{T}\boldsymbol{T}_{\lambda
}(G)-\boldsymbol{m}\right)  ^{T}(\boldsymbol{M}^{T}\boldsymbol{\Sigma
}_{\lambda}(\widehat{\boldsymbol{\beta}}_{\lambda})\boldsymbol{M})^{-1}\left(
\boldsymbol{M}^{T}\boldsymbol{T}_{\lambda}(G)-\boldsymbol{m}\right)
.\label{WoF}%
\end{equation}

Now, considering the $\varepsilon$-contaminated joint distribution
$G_{\varepsilon}=(1-\varepsilon)G+\varepsilon\wedge_{\boldsymbol{w}}$ with
respect to the point mass contamination distribution $\wedge_{\boldsymbol{w}}$
at the contamination point $\boldsymbol{w}=(\boldsymbol{x}_{t},y_{t})$, the
influence function of $W_{\lambda}(\cdot)$ is defined as%
\begin{align*}
&  \mathcal{IF}(\boldsymbol{w},W_{\lambda},G)=\left.  \frac{\partial
W_{\lambda}(G_{\varepsilon})}{\partial\varepsilon}\right\vert _{\varepsilon
=0}\\
&  =\left(  \boldsymbol{M}^{T}T_{\lambda}(G)-\boldsymbol{m}\right)
^{T}\left(  \boldsymbol{M}^{T}\boldsymbol{\Sigma}_{\lambda}\left(
\boldsymbol{\beta}_{0}\right)  \boldsymbol{M}\right)  ^{-1}\boldsymbol{M}%
^{T}\mathcal{IF}(\boldsymbol{w},T_{\lambda},G).
\end{align*}
Now, assuming the null hypothesis to be true, let $G_{0}$ denote the joint
model distribution with true parameter value $\boldsymbol{\beta}_{0}$
satisfying $\boldsymbol{M}^{T}\boldsymbol{\beta}_{0}=\boldsymbol{m}$. Then,
under $G_{0}$, we have $\boldsymbol{T}_{\lambda}(G_{0})=\boldsymbol{\beta}%
_{0}$ and hence $\mathcal{IF}(\boldsymbol{w},W_{\lambda},G_{0})=\boldsymbol{0}%
$. Therefore, the first order influence function analysis is not adequate to
quantify the robustness of the proposed Wald-type test statistics $W_{\lambda
}$. It is bounded in the contamination points $\boldsymbol{w}=(\boldsymbol{x}%
_{t},y_{t})$ for all $\lambda\geq0$ but does not necessarily imply the
robustness of the tests since it includes the well-known non-robust MLE based
Wald-test at $\lambda=0$. This fact is consistent with the robustness analysis
of different other Wald-type tests under different set-ups (See, for example,
Rousseeuw and Ronchetti, 1979; Toma and Broniatowski, 2011; Ghosh et al., 2016
etc.) and we need to consider the second order influence analysis to asses the
robustness of $W_{\lambda}$.

The second order influence function of the Wald-type test statistics $W_{n}$
at the joint distribution $G$ is defined as
\begin{align*}
&  \mathcal{IF}_{2}(\boldsymbol{w},W_{\lambda},G)=\left.  \frac{\partial
^{2}W_{\lambda}(G_{\varepsilon})}{\partial\varepsilon^{2}}\right\vert
_{\varepsilon=0}\\
&  =\left(  \boldsymbol{M}^{T}\boldsymbol{T}_{\lambda}(G)-\boldsymbol{m}%
\right)  ^{T}\left(  \boldsymbol{M}^{T}\boldsymbol{\Sigma}_{\lambda}\left(
\boldsymbol{\beta}\right)  \boldsymbol{M}\right)  ^{-1}\boldsymbol{M}%
^{T}\mathcal{IF}_{2}(\boldsymbol{w},\boldsymbol{T}_{\lambda},G)\\
&  +\mathcal{IF}^{T}(\boldsymbol{w},\boldsymbol{T}_{\lambda},G)\boldsymbol{M}%
\left(  \boldsymbol{M}^{T}\boldsymbol{\Sigma}_{\lambda}\left(
\boldsymbol{\beta}\right)  \boldsymbol{M}\right)  ^{-1}\boldsymbol{M}%
^{T}\mathcal{IF}(\boldsymbol{w},\boldsymbol{T}_{\lambda},G).
\end{align*}
Again, under the null hypothesis $H_{0}$ with $\boldsymbol{\beta}_{0}$ being
the corresponding true parameter value, this second order influence function
simplifies further as presented in the following theorem and yields the
possibility to study the robustness of our proposed tests through its boundedness.

\begin{theorem}
\label{THM:second_IT_test} The second order influence function of the proposed
Wald-type test statistics $W_{n}$, given in Definition 5, at the null model
distribution $G_{0}$ having true parameter value $\boldsymbol{\beta}_{0}$ is
given by
\begin{align*}
&  \mathcal{IF}_{2}(\boldsymbol{w},W_{\lambda},G_{0})\\
&  =\mathcal{IF}^{T}(\boldsymbol{w},\boldsymbol{T}_{\lambda},G_{0}%
)\boldsymbol{M}\left(  \boldsymbol{M}^{T}\boldsymbol{\Sigma}_{\lambda}\left(
\boldsymbol{\beta}_{0}\right)  \boldsymbol{M}\right)  ^{-1}\boldsymbol{M}%
^{T}\mathcal{IF}(\boldsymbol{w},\boldsymbol{T}_{\lambda},G_{0}).\\
&  =\widetilde{\Psi}_{\lambda}^{2}(\boldsymbol{x}_{t}^{T}\boldsymbol{\beta
}_{0},y_{t})\boldsymbol{x}_{t}^{T}\boldsymbol{J}_{\lambda}^{-1}%
(\boldsymbol{\beta}_{0})\boldsymbol{M}\left(  \boldsymbol{M}^{T}%
\boldsymbol{\Sigma}_{\lambda}\left(  \boldsymbol{\beta}_{0}\right)
\boldsymbol{M}\right)  ^{-1}\boldsymbol{M}^{T}\boldsymbol{J}_{\lambda}%
^{-1}(\boldsymbol{\beta}_{0})\boldsymbol{x}_{t}.
\end{align*}

\end{theorem}

Note that, the influence function of the Wald-type test statistic is directly
a quadratic function of the corresponding MDPDE used. Hence, as described in
the previous subsection, the influence function for the proposed tests with
$\lambda>0$ will be small and bounded for all kinds of outliers in a logistic
regression model, whereas the classical MLE based Wald-type test will have an
unbounded influence function for large \textquotedblleft bad" leverage points.
Figure \ref{FIG:IF_testb1} shows the plots of this second order influence
functions for the Wald-type test statistics for different $\lambda$ for
testing the significance of the first slope parameter in a logistic regression
model with with two independent standard normal covariates and
$\boldsymbol{\beta}_{0}=(0,1,1)^{T}$ fixing $y_{t}=0$. The behavior of the
influence functions are again similar to those observed for the corresponding
$MDPDE$ in Figure \ref{FIG:IF_testb1}, which shows the greater robustness of
our proposal at larger positive $\lambda$ over the non-robust MLE based Wald
test at $\lambda=0$.%

\begin{figure}[htbp]  \centering
$\
\begin{tabular}
[c]{cc}%
{\includegraphics[
trim=0.000000in 0.000000in 0.418824in 0.000000in,
height=2.6515in,
width=3.2768in
]%
{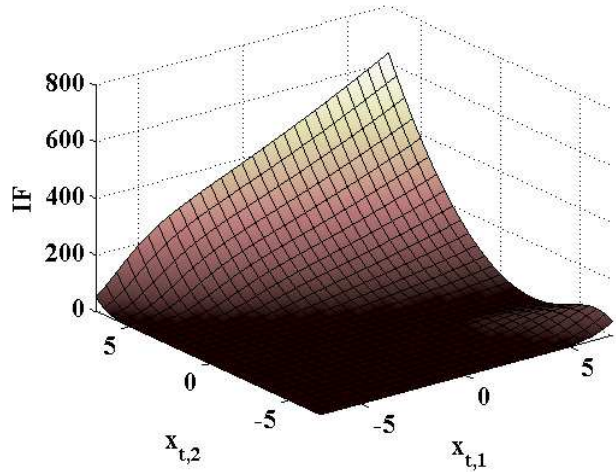}%
}
&
{\includegraphics[
trim=0.000000in 0.000000in 0.418824in 0.000000in,
height=2.6515in,
width=3.2768in
]%
{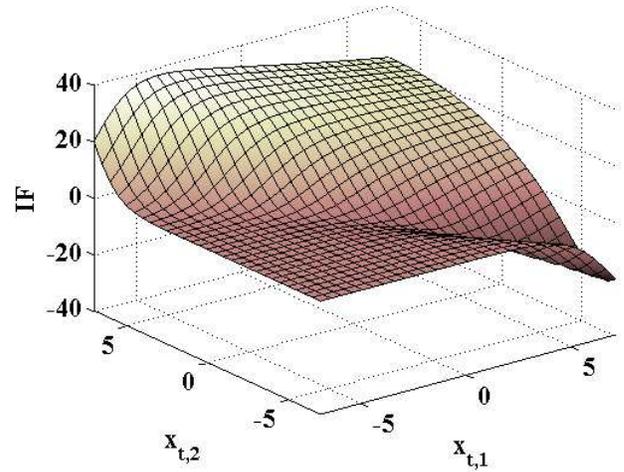}%
}
\\
(a) $\lambda=0$ & (b) $\lambda=0.1$\\%
{\includegraphics[
trim=0.000000in 0.000000in 0.418824in 0.000000in,
height=2.6515in,
width=3.2768in
]%
{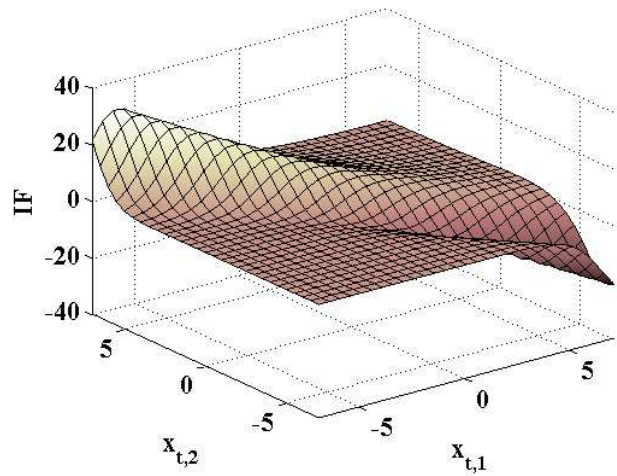}%
}
&
{\includegraphics[
trim=0.000000in 0.000000in 0.418824in 0.000000in,
height=2.6515in,
width=3.2768in
]%
{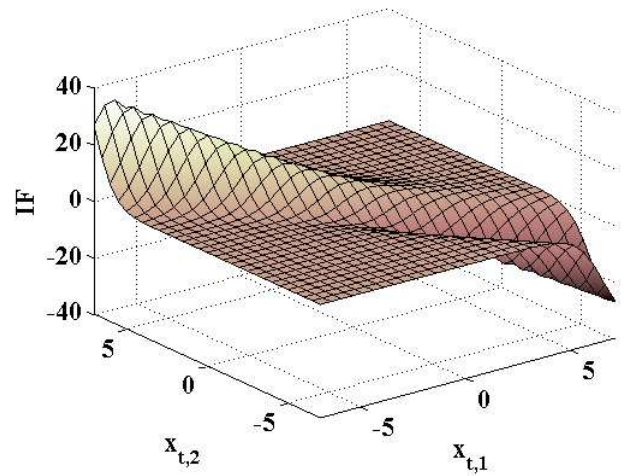}%
}
\\
(c) $\lambda=0.5$ & (d) $\lambda=1$%
\end{tabular}
\ \ $%
\caption{Second order Influence function of the Wald-type test statistics
for testing significance of the first slope parameter $\protect \beta_1$ for
different $\protect \lambda$ ($y_t=0$).\label{FIG:IF_testb1}}%
\end{figure}%


\subsection{Level and Power Influence Functions}

We now study the robustness of the proposed tests through the stability of
their Type-I and Type-II error which are two basic components for measuring
the performance of any testing procedure. In particular, we will study eth
local stability of level and power of the proposed tests through corresponding
influence function analysis. Note that the finite sample level and power of
our proposed Wald-type tests are difficult to compute and has no general form;
on the other hand, the tests are consistent having asymptotic power as one
against any fixed alternative. So, we will study the influence function of the
asymptotic level under the null $\boldsymbol{\beta}=\boldsymbol{\beta}_{0}$
and asymptotic power under the sequence of contiguous alternatives
$\boldsymbol{\beta}_{n}=\boldsymbol{\beta}_{0}+n^{-1/2}\boldsymbol{d}$ as
defined in, for example, Hampel et al.~(1986) and Ghosh et al.~(2016) among
others. In particular, assuming the contamination proportion tends to zero at
the same rate as the contiguous alternatives approaches to the null, here we
consider the following contaminated joint distribution for the power stability
calculation as
\begin{equation}
G_{n,\varepsilon,\boldsymbol{w}}^{P}=(1-\tfrac{\varepsilon}{\sqrt{n}%
})G_{\boldsymbol{\beta}_{n}}+\tfrac{\varepsilon}{\sqrt{n}}\wedge
_{\boldsymbol{w}},\label{EQ:power_F}%
\end{equation}
where $\boldsymbol{w}$ denote the contamination point $\boldsymbol{w}%
=(\boldsymbol{x}_{t}^{T},y_{t})^{T}$, and $G_{\boldsymbol{\beta}_{n}}$ denote
the joint model distribution with true parameter value $\boldsymbol{\beta
}=\boldsymbol{\beta}_{n}$. The contamination distribution to be considered for
the level stability check can be obtained by substituting $\boldsymbol{d}%
=\boldsymbol{0}$ in (\ref{EQ:power_F}), which yields
\[
G_{n,\varepsilon,\boldsymbol{w}}^{P}=(1-\tfrac{\varepsilon}{\sqrt{n}%
})G_{\boldsymbol{\beta}_{0}}+\tfrac{\varepsilon}{\sqrt{n}}\wedge
_{\boldsymbol{w}}.
\]
Then, the level and power influence functions are defined in terms of the
following quantities
\[
\alpha(\varepsilon,\boldsymbol{w})=\lim\limits_{n\rightarrow\infty
}P_{G_{n,\varepsilon,\boldsymbol{w}}^{L}}(W_{n}>\chi_{r,\alpha}^{2})\text{,}%
\]
and
\[
\pi(\boldsymbol{\beta}_{n},\varepsilon,\boldsymbol{x})=\lim
\limits_{n\rightarrow\infty}P_{G_{n,\varepsilon,\boldsymbol{w}}^{P}}%
(W_{n}>\chi_{r,\alpha}^{2}).
\]

\begin{definition}
The level influence function (LIF) and the power influence function (PIF) for
the Wald-type test statistics $W_{n}$ are defined respectively as
\[
\mathcal{LIF}(\boldsymbol{w};W_{n},G_{\boldsymbol{\beta}_{0}})=\left.
\dfrac{\partial}{\partial\varepsilon}\alpha(\varepsilon,\boldsymbol{w}%
)\right\vert _{\varepsilon=0},\quad\mathcal{PIF}(\mathbf{\boldsymbol{x}}%
;W_{n},G_{\boldsymbol{\beta}_{0}})=\left.  \dfrac{\partial}{\partial
\varepsilon}\pi(\boldsymbol{\beta}_{n},\varepsilon,\boldsymbol{w})\right\vert
_{\varepsilon=0}.
\]

\end{definition}

See Ghosh et al.~(2016) for an extensive discussion on the interpretations of
the level and power influence functions and their relations with the influence
function of the test statistics in the context of a general Wald-type test.

Next, we will derive the forms of the LIF and PIF for our proposed tests in
logistic regression model assuming the conditions required for the derivation
of asymptotic distributions of the MDPDE hold.

\begin{theorem}
\label{THM:7asymp_power_one} Assume that the conditions of Theorem 6 holds and
consider the contiguous alternatives $\boldsymbol{\beta}_{n}=\boldsymbol{\beta
}_{0}+n^{-1/2}\boldsymbol{d}$ along with the contaminated model in
(\ref{EQ:power_F}). Then we have the following results:

\begin{enumerate}
\item[(i)] The asymptotic distribution of the test statistics $W_{n}$ under
$G_{n,\varepsilon,\boldsymbol{w}}^{P}$ is non-central chi-square with $r$
degrees of freedom and the non-centrality parameter
\[
\delta=\widetilde{\boldsymbol{d}}_{\varepsilon,\boldsymbol{w},\lambda}%
^{T}(\boldsymbol{\beta}_{0})\boldsymbol{M}\left(  \boldsymbol{M}%
^{T}\boldsymbol{\Sigma}_{\lambda}\left(  \boldsymbol{\beta}_{0}\right)
\boldsymbol{M}\right)  ^{-1}\boldsymbol{M}^{T}\widetilde{\boldsymbol{d}%
}_{\varepsilon,\boldsymbol{w},\lambda}(\boldsymbol{\beta}_{0}),
\]
where $\widetilde{\boldsymbol{d}}_{\varepsilon,\boldsymbol{w},\lambda
}(\boldsymbol{\beta}_{0})=\boldsymbol{d}+\varepsilon\mathcal{IF}%
(\boldsymbol{w},\boldsymbol{T}_{\lambda},G_{\boldsymbol{\beta}_{0}}).$

\item[(ii)] The asymptotic power under $G_{n,\varepsilon,\boldsymbol{w}}^{P} $
can be approximated as
\begin{align}
\pi(\boldsymbol{\beta}_{n},\varepsilon,\boldsymbol{w}) &  \cong P\left(
\chi_{r}^{2}(\delta)>\chi_{r,\alpha}^{2}\right) \nonumber\\
&  \cong\sum\limits_{v=0}^{\infty}C_{v}\left(  \boldsymbol{M}^{T}%
\widetilde{\boldsymbol{d}}_{\varepsilon,\boldsymbol{w},\lambda}%
(\boldsymbol{\beta}_{0}),\left(  \boldsymbol{M}^{T}\boldsymbol{\Sigma
}_{\lambda}\left(  \boldsymbol{\beta}_{0}\right)  \boldsymbol{M}\right)
^{-1}\right)  P\left(  \chi_{r+2v}^{2}>\chi_{r,\alpha}^{2}\right)  ,
\end{align}
where
\[
C_{v}\left(  \boldsymbol{t},\boldsymbol{A}\right)  =\frac{\left(
\boldsymbol{t}^{T}\boldsymbol{At}\right)  ^{v}}{v!2^{v}}e^{-\frac{1}%
{2}\boldsymbol{t}^{T}\boldsymbol{At}},
\]
$\chi_{p}^{2}(\delta)$ denotes a non-central chi-square random variable with
$p$ degrees of freedom and $\delta$ as non-centrality parameter and $\chi
_{q}^{2}=\chi_{q}^{2}(0)$ denotes a central chi-square random variable having
degrees of freedom $q$.
\end{enumerate}
\end{theorem}

\begin{proof}
Let us denote $\boldsymbol{\beta}_{n}^{\ast}=\boldsymbol{T}_{\lambda
}(G_{n,\varepsilon,\boldsymbol{w}}^{P})$. Then, we get
\begin{align}
W_{n} &  =n(\boldsymbol{M}^{T}\widehat{\boldsymbol{\beta}}_{\lambda
}-\boldsymbol{m})^{T}\left(  \boldsymbol{M}^{T}\boldsymbol{\Sigma}_{\lambda
}\left(  \boldsymbol{\beta}_{0}\right)  \boldsymbol{M}\right)  ^{-1}%
(\boldsymbol{M}^{T}\widehat{\boldsymbol{\beta}}_{\lambda}-\boldsymbol{m}%
)\nonumber\\
&  =n\left(  \boldsymbol{M}^{T}\boldsymbol{\beta}_{n}^{\ast}-\boldsymbol{m}%
\right)  ^{T}\left(  \boldsymbol{M}^{T}\boldsymbol{\Sigma}_{\lambda}\left(
\boldsymbol{\beta}_{0}\right)  \boldsymbol{M}\right)  ^{-1}\left(
\boldsymbol{M}^{T}\boldsymbol{\beta}_{n}^{\ast}-\boldsymbol{m}\right)
\nonumber\\
&  +n(\widehat{\boldsymbol{\beta}}_{\lambda}-\boldsymbol{\beta}_{n}^{\ast
})^{T}\boldsymbol{M}\left(  \boldsymbol{M}^{T}\boldsymbol{\Sigma}_{\lambda
}\left(  \boldsymbol{\beta}_{0}\right)  \boldsymbol{M}\right)  ^{-1}%
\boldsymbol{M}^{T}(\widehat{\boldsymbol{\beta}}_{\lambda}-\boldsymbol{\beta
}_{n}^{\ast})\nonumber\\
&  +n(\widehat{\boldsymbol{\beta}}_{\lambda}-\boldsymbol{\beta}_{n}^{\ast
})^{T}\boldsymbol{M}\left(  \boldsymbol{M}^{T}\boldsymbol{\Sigma}_{\lambda
}\left(  \boldsymbol{\beta}_{0}\right)  \boldsymbol{M}\right)  ^{-1}%
(\boldsymbol{M}^{T}\widehat{\boldsymbol{\beta}}_{\lambda}-\boldsymbol{m}%
)\nonumber\\
&  =S_{1,n}+S_{2,n}+S_{3,n}.\label{EQ:eq0}%
\end{align}
Next, one can show that
\begin{align}
\sqrt{n}(\boldsymbol{\beta}_{n}^{\ast}-\boldsymbol{\beta}_{0}) &
=\boldsymbol{d}+\varepsilon\mathcal{IF}\left(  \boldsymbol{w},\boldsymbol{T}%
_{\lambda},G_{\boldsymbol{\beta}_{0}}\right)  +o_{p}(\boldsymbol{1}%
_{p})\nonumber\\
&  =\widetilde{\boldsymbol{d}}_{\varepsilon,\boldsymbol{w},\lambda
}(\boldsymbol{\theta}_{0})+o_{p}(\boldsymbol{1}_{p}).\label{CDT}%
\end{align}
Thus, we get
\begin{equation}
\sqrt{n}(\boldsymbol{M}^{T}\boldsymbol{\beta}_{n}^{\ast}-\boldsymbol{m}%
)=\boldsymbol{M}^{T}\widetilde{\boldsymbol{d}}_{\varepsilon,\boldsymbol{w}%
,\lambda}(\boldsymbol{\theta}_{0})+o_{p}(\boldsymbol{1}_{p}).\label{291}%
\end{equation}
Further, under $G_{n,\varepsilon,\boldsymbol{w}}^{P}$, the asymptotic
distribution of MDPDE yields
\begin{equation}
\sqrt{n}(\widehat{\boldsymbol{\beta}}_{\lambda}-\boldsymbol{\beta}_{n}^{\ast
})\underset{n\rightarrow\infty}{\overset{\mathcal{L}}{\longrightarrow}%
}\mathcal{N}\left(  \boldsymbol{0},\boldsymbol{\Sigma}_{\lambda}\left(
\boldsymbol{\beta}_{0}\right)  \right)  .\label{292}%
\end{equation}
Thus, we get
\[
S_{3,n}\underset{n\rightarrow\infty}{\overset{\mathcal{L}}{\longrightarrow}%
}\chi_{r}^{2}.
\]
Combining (\ref{EQ:eq0}), (\ref{291}) and (\ref{292}), we get
\[
W_{n}=\boldsymbol{Z}_{n}^{T}\left(  \boldsymbol{M}^{T}\boldsymbol{\Sigma
}_{\lambda}\left(  \boldsymbol{\beta}_{0}\right)  \boldsymbol{M}\right)
^{-1}\boldsymbol{Z}_{n}+o_{p}(1),
\]
where
\[
\boldsymbol{Z}_{n}=\sqrt{n}\boldsymbol{M}^{T}(\widehat{\boldsymbol{\beta}%
}_{\lambda}-\boldsymbol{\beta}_{n}^{\ast})+\boldsymbol{M}^{T}%
\widetilde{\boldsymbol{d}}_{\varepsilon,\boldsymbol{w},\lambda}%
(\boldsymbol{\theta}_{0}).
\]
By (\ref{292}),
\[
\boldsymbol{Z}_{n}\underset{n\rightarrow\infty}{\overset{\mathcal{L}%
}{\longrightarrow}}\mathcal{N}\left(  \boldsymbol{M}^{T}%
\widetilde{\boldsymbol{d}}_{\varepsilon,\boldsymbol{w},\lambda}%
(\boldsymbol{\theta}_{0}),\boldsymbol{M}^{T}\boldsymbol{\Sigma}_{\lambda
}\left(  \boldsymbol{\beta}_{0}\right)  \boldsymbol{M}\right)  ,
\]
and hence we get that
\[
W_{n}\underset{n\rightarrow\infty}{\overset{\mathcal{L}}{\longrightarrow}}%
\chi_{r}^{2}(\delta),
\]
where $\delta$ is as defined in Part (i) of the theorem.\newline Part (ii) of
the theorem follows from Part (i) using the infinite series expansion of a
non-central distribution function in terms of that of the central chi-square
variables:
\begin{align}
\pi(\boldsymbol{\beta}_{n},\varepsilon,\boldsymbol{w}) &  =\lim_{n\rightarrow
\infty}P_{G_{n,\varepsilon,\boldsymbol{w}}^{P}}(W_{n}>\chi_{r,\alpha}%
^{2})\cong P(\chi_{r,\delta}^{2}>\chi_{r,\alpha}^{2})\nonumber\\
&  =\sum\limits_{v=0}^{\infty}C_{v}\left(  \boldsymbol{M}^{T}%
\widetilde{\boldsymbol{d}}_{\varepsilon,\boldsymbol{w},\lambda}%
(\boldsymbol{\beta}_{0}),\left(  \boldsymbol{M}^{T}\boldsymbol{\Sigma
}_{\lambda}\left(  \boldsymbol{\beta}_{0}\right)  \boldsymbol{M}\right)
^{-1}\right)  P\left(  \chi_{r+2v}^{2}>\chi_{r,\alpha}^{2}\right)  .\nonumber
\end{align}

\end{proof}

\begin{corollary}
\label{COR:7cont_power_one} Putting $\varepsilon=0$ in Theorem
\ref{THM:7asymp_power_one}, we get the asymptotic power of the proposed
Wald-type tests under the contiguous alternative hypotheses $\boldsymbol{\beta
}_{n}=\boldsymbol{\beta}_{0}+n^{-1/2}\boldsymbol{d}$ as
\begin{equation}
\pi(\boldsymbol{\beta}_{n})=\pi(\boldsymbol{\beta}_{n},0,\boldsymbol{w}%
)\cong\sum\limits_{v=0}^{\infty}C_{v}\left(  \boldsymbol{M}^{T}\boldsymbol{d}%
,\left(  \boldsymbol{M}^{T}\boldsymbol{\Sigma}_{\lambda}\left(
\boldsymbol{\beta}_{0}\right)  \boldsymbol{M}\right)  ^{-1}\right)  P\left(
\chi_{r+2v}^{2}>\chi_{r,\alpha}^{2}\right)  .\nonumber
\end{equation}
This is identical with the results obtained earlier in Theorem 10 independently.
\end{corollary}

\begin{corollary}
\label{COR:7asymp_level_one} Putting $\boldsymbol{d}=\boldsymbol{0}$ in
Theorem \ref{THM:7asymp_power_one}, we get the asymptotic distribution of
$W_{n}$ under $G_{n,\varepsilon,\boldsymbol{w}}^{L}$ as the non-central
chi-square distribution having $r$ degrees of freedom and non-centrality
parameter
\[
\varepsilon^{2}\mathcal{IF}(\boldsymbol{w};\boldsymbol{T}_{\lambda
},G_{\boldsymbol{\beta}_{0}})^{T}\boldsymbol{M}\left(  \boldsymbol{M}%
^{T}\boldsymbol{\Sigma}_{\lambda}\left(  \boldsymbol{\beta}_{0}\right)
\boldsymbol{M}\right)  ^{-1}\boldsymbol{M}^{T}\mathcal{IF}(\boldsymbol{w}%
;\boldsymbol{T}_{\lambda},G_{\boldsymbol{\beta}_{0}}).
\]
Then, the asymptotic level under contiguous contamination is given by
\begin{align}
&  \alpha(\varepsilon,\boldsymbol{w})=\pi(\boldsymbol{\beta}_{0}%
,\varepsilon,\boldsymbol{w})\nonumber\\
&  \cong\sum\limits_{v=0}^{\infty}C_{v}\left(  \varepsilon\boldsymbol{M}%
^{T}\mathcal{IF}(\boldsymbol{w};\boldsymbol{T}_{\lambda},G_{\boldsymbol{\beta
}_{0}}),\left(  \boldsymbol{M}^{T}\boldsymbol{\Sigma}_{\lambda}\left(
\boldsymbol{\beta}_{0}\right)  \boldsymbol{M}\right)  ^{-1}\right)  P\left(
\chi_{r+2v}^{2}>\chi_{r,\alpha}^{2}\right)  .\nonumber
\end{align}
In particular, as $\varepsilon\rightarrow0,\boldsymbol{\beta}_{n}^{\ast
}\rightarrow\boldsymbol{\beta}_{0}$ and the non-centrality parameter of the
above asymptotic distribution tends to zero leading to the null distribution
of $W_{n}$.
\end{corollary}

Now we can easily obtain the the power and level influence functions of the
Wald-type test statistics from Theorem \ref{THM:7asymp_power_one} and
Corollary \ref{COR:7asymp_level_one} and these have been presented in the
following theorem.

\begin{theorem}
\label{Theorem10} Under the assumptions of Theorem \ref{THM:7asymp_power_one},
the power and level influence functions of the proposed Wald-type test
statistic $W_{n}$ is given by
\begin{equation}
\mathcal{PIF}(\boldsymbol{w},W_{n},G_{\boldsymbol{\beta}_{0}})\cong
K_{r}^{\ast}\left(  \boldsymbol{s}^{T}\left(  \boldsymbol{\beta}_{0}\right)
\mathbf{d}\right)  \boldsymbol{s}^{T}(\boldsymbol{\beta}_{0})\mathcal{IF}%
(\boldsymbol{w},\boldsymbol{T}_{\lambda},G_{\boldsymbol{\beta}_{0}%
}),\label{EQ:7PIF_simpleTest1}%
\end{equation}
with $\boldsymbol{s}^{T}(\boldsymbol{\beta}_{0})=\boldsymbol{d}^{T}%
\boldsymbol{M}\left(  \boldsymbol{M}^{T}\boldsymbol{\Sigma}_{\lambda}\left(
\boldsymbol{\beta}_{0}\right)  \boldsymbol{M}\right)  ^{-1}\boldsymbol{M}^{T}$
and
\[
K_{r}^{\ast}(s)=e^{-\frac{s}{2}}\sum\limits_{v=0}^{\infty}\frac{s^{v-1}%
}{v!2^{v}}\left(  2v-s\right)  P\left(  \chi_{r+2v}^{2}>\chi_{r,\alpha}%
^{2}\right)  ,
\]
and
\[
\mathcal{LIF}(\boldsymbol{w},W_{n},G_{\boldsymbol{\beta}_{0}})=0.
\]
Further, the derivative of $\alpha(\varepsilon,\boldsymbol{w})$ of any order
with respect to $\varepsilon$ will be zero at $\varepsilon=0$, implying that
the level influence function of any order will be zero.
\end{theorem}

\begin{proof}
We start with the expression of $\pi(\boldsymbol{\beta}_{n},\varepsilon
,\boldsymbol{w})$ from Theorem \ref{THM:7asymp_power_one}. Clearly, by
definition of PIF and using the chain rule of derivatives, we get
\begin{align*}
&  \mathcal{PIF}(\boldsymbol{w},W_{n},G_{\boldsymbol{\beta}_{0}}%
)=\frac{\partial}{\partial\varepsilon}\left.  \pi(\boldsymbol{\beta}%
_{n},\varepsilon,\boldsymbol{w})\right\vert _{\varepsilon=0}\\
&  \cong\sum\limits_{v=0}^{\infty}\frac{\partial}{\partial\varepsilon}\left.
C_{v}\left(  \boldsymbol{M}^{T}\widetilde{\boldsymbol{d}}_{\varepsilon
,\boldsymbol{w},\lambda}(\boldsymbol{\beta}_{0}),\left(  \boldsymbol{M}%
^{T}\boldsymbol{\Sigma}_{\lambda}\left(  \boldsymbol{\beta}_{0}\right)
\boldsymbol{M}\right)  ^{-1}\right)  \right\vert _{\varepsilon=0}P\left(
\chi_{r+2v}^{2}>\chi_{r,\alpha}^{2}\right) \\
&  \cong\sum\limits_{v=0}^{\infty}\frac{\partial}{\partial\boldsymbol{t}^{T}%
}\left.  C_{v}\left(  \boldsymbol{M}^{T}\boldsymbol{t},\left(  \boldsymbol{M}%
^{T}\boldsymbol{\Sigma}_{\lambda}\left(  \boldsymbol{\beta}_{0}\right)
\boldsymbol{M}\right)  ^{-1}\right)  \right\vert _{\boldsymbol{t}%
=\widetilde{\boldsymbol{d}}_{0,\boldsymbol{w},\lambda}(\boldsymbol{\beta}%
_{0})}\frac{\partial}{\partial\varepsilon}\left.  \widetilde{\boldsymbol{d}%
}_{\varepsilon,\boldsymbol{w},\lambda}(\boldsymbol{\beta}_{0})\right\vert
_{\varepsilon=0}P\left(  \chi_{r+2v}^{2}>\chi_{r,\alpha}^{2}\right)  .
\end{align*}
Now $\widetilde{\boldsymbol{d}}_{0,\boldsymbol{w},\lambda}(\boldsymbol{\beta
}_{0})=\boldsymbol{d}$ and standard differentiations give
\[
\frac{\partial}{\partial\varepsilon}\widetilde{\boldsymbol{d}}_{\varepsilon
,\boldsymbol{w},\lambda}(\boldsymbol{\beta}_{0})=\mathcal{IF}(\boldsymbol{w}%
,\boldsymbol{T}_{\lambda},G_{\boldsymbol{\beta}_{0}}),
\]
and
\[
\frac{\partial}{\partial\boldsymbol{t}}C_{v}\left(  \boldsymbol{t}%
,\boldsymbol{A}\right)  =\frac{\left(  \boldsymbol{t}^{T}\boldsymbol{At}%
\right)  ^{v-1}}{v!2^{v}}\left(  2v-\boldsymbol{t}^{T}\boldsymbol{At}\right)
\boldsymbol{At}e^{-\frac{1}{2}\boldsymbol{t}^{T}\boldsymbol{At}}.
\]
Combining above results and simplifying, we get the required expression of PIF
as presented in the theorem.
\end{proof}

\bigskip It is clear from the above theorem that, the asymptotic level of the
proposed Wald-type test statistic will be unaffected by a contiguous
contamination for any values of the tuning parameter $\lambda$, whereas the
power influence function will be bounded whenever the influence function of
the MDPDE is bounded (which happens for all $\lambda>0$). Thus, the robustness
of the power of the proposed tests again turns out to be directly dependent on
the robustness of the MDPDE $\boldsymbol{\beta}_{\lambda}$ used in
constructing the test. In particular, the asymptotic contiguous power of the
classical MLE based Wald-type test (at $\lambda=0$) will be non-robust whereas
that for the Wald-type tests with $\lambda>0$ will be robust under contiguous
contaminations and this robustness increases as $\lambda$ increases further.

\section{Simulation study\label{sec5}}%

\begin{figure}[htbp]  \centering
$\
\begin{tabular}
[c]{cc}%
{\includegraphics[
height=3.2448in,
width=3.2448in
]%
{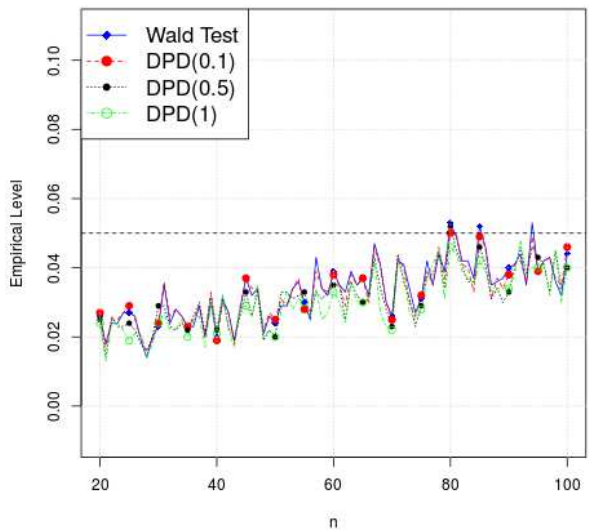}%
}
&
{\includegraphics[
height=3.2448in,
width=3.2448in
]%
{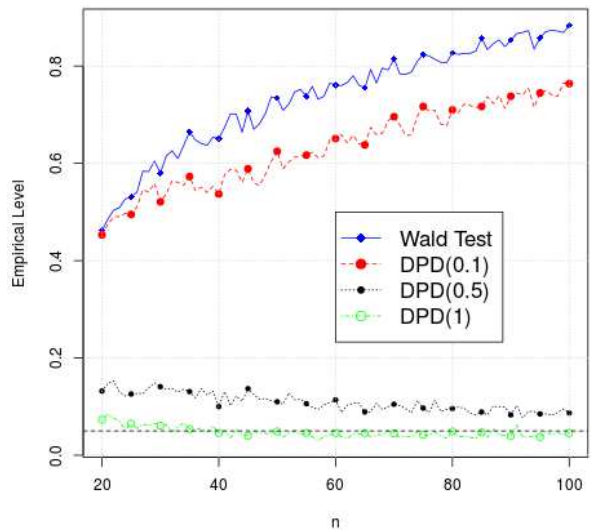}%
}
\\
(a) & (b)\\%
{\includegraphics[
height=3.2448in,
width=3.2448in
]%
{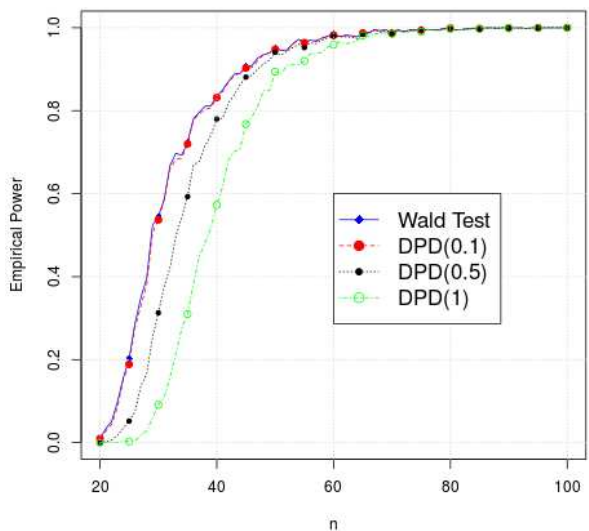}%
}
&
{\includegraphics[
height=3.2448in,
width=3.2448in
]%
{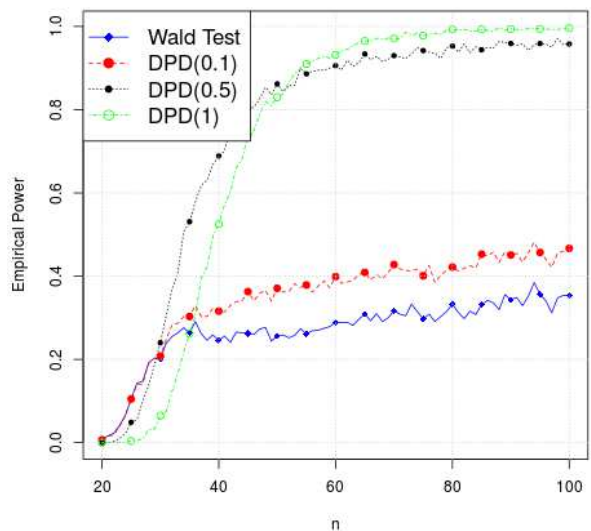}%
}
\\
(c) & (d)
\end{tabular}
\ \ $\caption{(a) Simulated levels of different tests for pure data; (b)
simulated levels of different tests for contaminated data; (c) simulated
powers of different tests for pure data; (d) simulated powers of different
tests for contaminated data.\label{fig:simulation}}%
\end{figure}%

In this section we have empirically demonstrated some of the strong robustness
properties of the density power divergence tests for the logistic regression
model. We considered two explanatory variables $x_{1}$ and $x_{2}$ in this
study, so $k=2$. These two variables are distributed according a standard
normal distribution $\mathcal{N}(\boldsymbol{0}{,}\boldsymbol{I}_{2\times2})$.
The response variables $Y_{i}$ are generated following the logit model as
given in (\ref{logit}). The true value of the parameter is taken as
$\boldsymbol{\beta}_{0}=(0,1,1)^{T}$. We considered the null hypothesis
$H_{0}:(\beta_{1},\beta_{2})^{T}=(1,1)^{T}$. It can be written in the form of
the general hypothesis given in (\ref{1.2}), where $\boldsymbol{m}=(1,1)^{T}$
and
\[
\boldsymbol{M}=\left(
\begin{array}
[c]{cc}%
0 & 0\\
1 & 0\\
0 & 1
\end{array}
\right)  .
\]
Our interest was in studying the observed level (measured as the proportion of
test statistics exceeding the corresponding chi-square critical value in a
large number -- here $1000$ -- of replications) of the test under the correct
null hypothesis. The result is given in Figure \ref{fig:simulation}(a) where
the sample size $n$ varies from 20 to 100. We have used several Wald-type test
statistics, corresponding to different minimum density power divergnece
estimators. We have used, $\lambda=0,\ 0.1,\ 0.5$ and $1$, in this particular
study. As it is previously mentioned, $\lambda=0$ is the classical Wald test
for the logistic regression model. The horizontal lines in the figure
represents the nominal level of 0.05. It may be noticed that all the tests are
slightly conservative for small sample sizes and lead to somewhat deflated
observed levels. In particular, the Wald-type tests with higher values of
$\lambda$ are relatively more conservative. However, this discrepancy
decreases rapidly as sample size increases.

To evaluate the stability of the level of the tests under contamination, we
repeated the tests for the same null hypothesis by adding 3\% outliers in the
data. For the outlying observations we first introduced the leverage points
where $x_{1}$ and $x_{2}$ are generated from $\mathcal{N}(\boldsymbol{\mu}%
_{c},\sigma\boldsymbol{I}_{2\times2})$ with $\boldsymbol{\mu}_{c}=(5,5)^{T}$
and $\sigma=0.01$. Then the values of the response variable corresponding to
those leverage points were altered to produce vertical outliers ($y_{t}=1$ was
converted to $y_{t}=0$). Figure \ref{fig:simulation}(b) shows that the levels
of the classical Wald test as well as DPD(0.1) test break down, whereas
Wald-type test statistics for $\lambda=0.5$ and $\lambda=1$ present highly
stable levels.

To investigate the power of the tests we changed the null hypothesis to
$H_{0}^{\ast}:(\beta_{1},\beta_{2})^{T}=(0,0)^{T}$, and kept the data
generating distributions as before, as well as the true value of the parameter
as $\boldsymbol{\beta}_{0}=(0,1,1)^{T}$. In terms of the null hypothesis in
(\ref{1.2}) the value of $\boldsymbol{m}$ is changed to $(0,0)^{T}$ whereas
$\boldsymbol{M}$ remained unchanged from the previous experiment. The
empirical power functions are calculated in the same manner as the levels of
the tests, and plotted in Figure \ref{fig:simulation}(c). The Wald test is the
most powerful under pure data. The power of the Wald-type test statistic for
$\lambda=0.1$ almost coincide with the classical Wald test in this case. The
performances of the Wald-type test statisdtics for $\lambda=0.5$ and
$\lambda=1$ are relatively poor, however, as the sample size increases to $60$
and beyond, the powers are practically identical.

Finally, we calculated the power functions under contamination for the above
hypothesis under the same setup as of the level contamination. The observed
powers of that the tests are given in Figure \ref{fig:simulation}(d). The
Wald-type test statistics for $\lambda=0.5$ and $\lambda=1$ show stable powers
under contamination, but the classical Wald test and the Wald-type test for
$\lambda=0.1$ exhibit a drastic loss in power. In very small sample sizes the
classical Wald test and the Wald-type test for $\lambda=0.1 $ have slightly
higher power than the other tests, but this must be a consequence of the
observed levels of these tests being higher than the latter for such sample
sizes. On the whole, the proposed Wald-type test statistics corresponding to
moderately large $\lambda$ appear to be quite competitive to the classical
Wald test for pure normal data, but they are far better in terms of robustness
properties under contaminated data.

\section{Real Data Examples\label{sec6}}

In this section we will explore the performance of the proposed Wald-type
tests in logistic regression models by applying it on different interesting
real data sets. The estimators are computed by minimizing the corresponding
density power divergence through the software \texttt{R}, and the minimization
is performed using \textquotedblleft optim\textquotedblright\ function.

\subsection{Students Data}

As an interesting data example leading to the logistic regression model, we
consider the students data set from Mu\~{n}oz-Garcia et al. (2006). The data
set consists of $576$ students of the University of Seville. The response
variable is the students aim to graduate after three years. The explanatory
variables are gender ($x_{i1}=0$ if male; $x_{i1}=1$ if female), entrance
examination (EE) in University ($x_{i2}=1$ if the first time; $x_{i2}=0$
otherwise) and sum of marks ($x_{i3}$) obtained for the courses of first term.
There were 61 distinct cases (i.e. $n=61$) in this study. We assume that the
response variable follows a binomial logistic regression model as mentioned in
Remark \ref{rem}. We are interested to test the null hypothesis that the
gender of student does not play any role on their aim. So the null hypothesis
is given by $H_{0}:\beta_{1}=0$. Figure \ref{fig:student} shows $p$-values of
Wald-type tests for different values of $\lambda$. Mu\~{n}oz-Garcia et al.
(2006) mentioned that 32nd observation is the most influential point as it has
a large residual and a high leverage value. If we use the classical Wald test
or Wald-type tests with small $\lambda$ under the full data, the null
hypothesis is rejected at $10\%$ level of significance. But this result is
clearly a false positive as the outlier deleted $p$-values for all $\lambda$
are close to $0.35$. On the other hand, Wald-type tests with large $\lambda$
give robust $p$-values in both situations.%

\begin{figure}[htbp]  \centering
$\
\begin{tabular}
[c]{c}%
$%
{\includegraphics[
trim=0.000000in 0.406024in 0.000000in 0.494130in,
height=1.7177in,
width=4.0108in
]%
{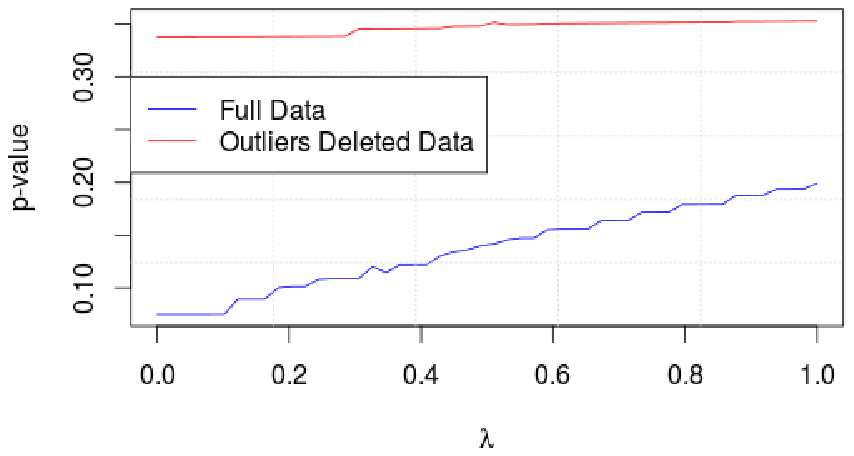}%
}
$%
\end{tabular}
\ $\caption{P-values of Wald-type tests for testing $H_0:\beta_1=0$ in
Students data.\label{fig:student}}%
\end{figure}%

\subsection{Lymphatic Cancer Data}

Brown (1980), Mart\'{\i}n and Pardo (2009) and Zelterman (2005, Section 3.3)
studied the data that focused on the evidence of lymphatic cancer in prostate
cancer patients for predicting lymph nodal involvement of cancer. There were
five covariates (three dichotomous and two continuous): the X-ray finding
($x_{i1}=1$ if present; $x_{i1}=0$ if absent), size of the tumor by palpation
($x_{i2}=1$ if serious; $x_{i2}=0$ if not serious), pathology grade by biopsy
($x_{i3}=1$ if serious; $x_{i3}=0$ if not serious), the age of the patient at
the time of diagnosis ($x_{i4}$) and serum acid phosphatase level ($x_{i5}$).
The diagnostics was associated with $53$ individuals. An ordinary logistic
model is assumed here. We are interested to test the significance of the size
of the tumor on the response variable, so the null hypothesis is taken as
$H_{0}:\beta_{2}=0$. The $p$-values of Wald-type tests for different values of
$\lambda$ are given in Figure \ref{fig:cancer}. Mart\'{\i}n and Pardo (2009)
noticed that the 24th observation is an influential point. The $p$-value of
the classical Wald test under the full data is $0.0430$, but if the outlier is
deleted it becomes $0.0668$. So if we consider a test at $5\%$ level of
significance, the decision of the test changes when we delete just one
outlying observation. However, Wald-type tests with high values of $\lambda$
always produce high $p$-values.%

\begin{figure}[htbp]  \centering
$\
\begin{tabular}
[c]{c}%
$%
{\includegraphics[
trim=0.000000in 0.406024in 0.000000in 0.493885in,
height=1.7186in,
width=4.0108in
]%
{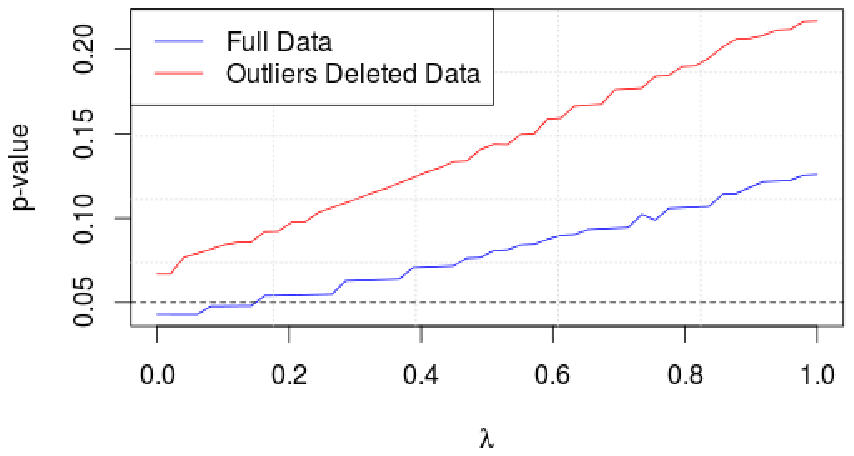}%
}
$%
\end{tabular}
\ $\caption{P-values of Wald-type tests for testing $H_0:\beta_2=0$ in
Lymphatic Cancer data.\label{fig:cancer}}%
\end{figure}%

\subsection{Vasoconstriction Data}

Finney (1947), Pregibon (1981) and Mart\'{\i}n and Pardo (2009) studied the
data where the interest is on the occurrence of vasoconstriction in the skin
of the finger. The covariates of the study were the logarithm of volume
($x_{i1}$) and the logarithm of rate ($x_{i2}$) of inspired air measured in
liters. Pregibon (1981) has shown that two observations, the 4th and 18th, are
not fitted well by the logistic model as they have large residuals. However,
it can be checked easily that these observations are only outliers in the
$y$-space and are not leverage points. Here we want to test that there is no
effect of the covariates, so the null hypothesis is given by $H_{0}:\beta
_{1}=\beta_{2}=0$. The $p$-value of the classical Wald test under the full
data is $0.0194$, and in the outlier deleted data it becomes $0.0371$. But,
Figure \ref{fig:vasoconstriction} shows that Wald-type tests with large
$\lambda$ produce large $p$-values.%

\begin{figure}[htbp]  \centering
$\
\begin{tabular}
[c]{c}%
$%
{\includegraphics[
trim=0.000000in 0.406024in 0.000000in 0.493885in,
natheight=2.447400in,
natwidth=3.646900in,
height=1.7186in,
width=4.0108in
]%
{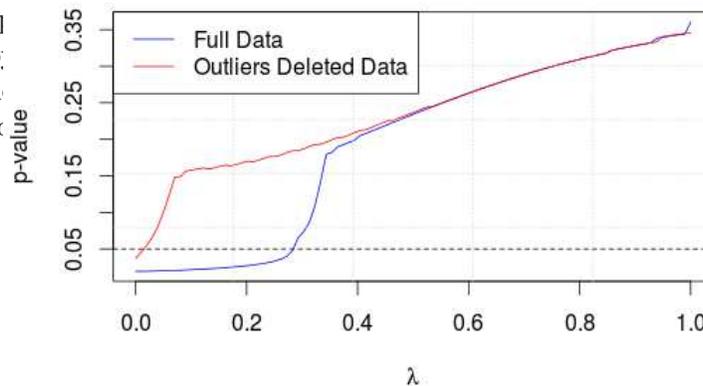}%
}
$%
\end{tabular}
\ $\caption{P-values of Wald-type tests for testing $H_0:\beta_1=\beta_2=0$
in Vasoconstriction data.\label{fig:vasoconstriction}}%
\end{figure}%

\subsection{Leukemia Data}

The data set consists of 33 cases on the survival of individuals diagnosed
with leukemia. The explanatory variables are white blood cell count ($x_{i1}$)
and another variable which indicates the presence or absence of a certain
morphologic characteristic in the white cells ($x_{i2}=1$ if present;
$x_{i2}=0$ if absent). This data set was also studied by Cook and Weisberg
(1982), Johnson (1985) and Mart\'{\i}n and Pardo (2009). They defined a
success to be patient survival in excess of 52 weeks. We are interested to
test the significance of two covariates, i.e. the null hypothesis is
$H_{0}:\beta_{1}=\beta_{2}=0$. The plot of the $p$-values of Wald-type tests
for different values of $\lambda$ is given in Figure \ref{fig:leukemia}.
Mart\'{\i}n and Pardo (2009) noticed that the 15th observation is an
influential point. The $p$-value of the classical Wald test under the full
data is $0.0226$, but if the outlier is deleted it becomes $0.0683$. Thus, at
$5\%$ level of significance, the decision of the test depends on only one
outlying observation. In this case also Wald-type tests with high values of
$\lambda$ always produce high $p$-values.%

\begin{figure}[htbp]  \centering
$\
\begin{tabular}
[c]{c}%
$%
{\includegraphics[
trim=0.000000in 0.405289in 0.000000in 0.494130in,
height=1.7186in,
width=4.0108in
]%
{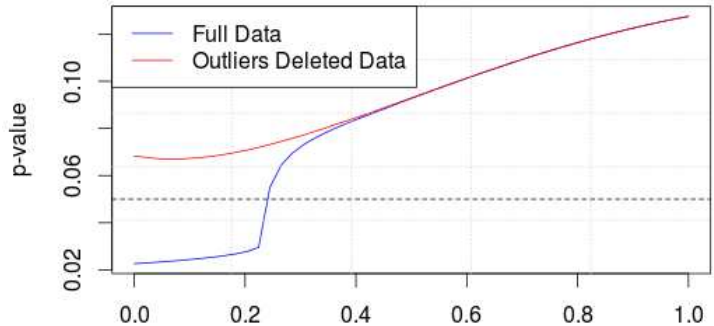}%
}
$%
\end{tabular}
\ $\caption{P-values of Wald-type tests for testing $H_0:\beta_1=\beta_2=0$
in Leukemia data.\label{fig:leukemia}}%
\end{figure}%

\section{Concluding Remarks}

Logistic regression for binary outcomes is one of the most popular and
successful tools in the statisticians toolbox. It is frequently used by
applied scientists of many disciplines to solve problems of real interest in
their doman of application. However, in the present age of big data, the need
for protection against data contamination and other modelling errors is
paramount, and, wherever possible, strong robustness qualities should be a
default requirement for statistical methods used in practice. In this paper we
have presented one such class of inference procedures. We have provided a
thorough theoretical evaluation of the proposted class of tests for testing
the linear hypothesis in the logistic regression model highlighting their
robustness advantages. We have also produced substantial numerical evidence,
including simulation results and a large number of real problems, to
demonstrate how these theoretical advantages translate in practice to real
gains. On the whole, we feel that the proposed tests will turn out to be an
useful method with significant practical application.

\textbf{Acknowledgement. }This research is partially supported by Grants
MTM2015-67057 and ECO2015-66593-P from Ministerio de Economia y Competitividad (Spain).

\end{document}